\documentclass[12pt,reqno]{amsart}
\DeclareUnicodeCharacter{3002}{.}
\usepackage[nolonger]{optional}
\usepackage{amssymb,bbm,graphicx}
\usepackage{color}
\usepackage{comment}
\usepackage{footnote}
\usepackage{enumitem}
\usepackage[colorlinks=true,linkcolor=black,citecolor=black,urlcolor=purple]{hyperref} 
\usepackage{flushend}
\usepackage{tikz}

\usepackage{amsmath}
\usepackage{mathtools}
\usepackage{subcaption}

\usepackage{cite}

\DeclareMathSymbol{\shortminus}{\mathbin}{AMSa}{"39} 

\usepackage{tikz}
\usepackage{tikz-3dplot}
\usepackage{tikz-cd}
\usepackage{xcolor}
\usetikzlibrary{arrows,decorations,decorations.markings,calc,positioning,shapes.geometric,arrows.meta} 
\usepackage[all]{xy}
\usepackage[margin=30mm]{geometry}
\usepackage{mathrsfs}
\usepackage{appendix}

\usepackage{setspace}
\newcommand{\marginparstretch}{0.6}
\let\oldmarginpar\marginpar
\renewcommand\marginpar[1]{\-\oldmarginpar[\framebox{\setstretch{\marginparstretch}\begin{minipage}{\marginparwidth}{\raggedleft\tiny #1}\end{minipage}}]{\framebox{\setstretch{\marginparstretch}\begin{minipage}{\marginparwidth}{\raggedright\tiny #1}\end{minipage}}}}

\numberwithin{equation}{section}
\numberwithin{figure}{section}
\theoremstyle{plain}
\newtheorem{theorem}{Theorem}[section]
\newtheorem{proposition}[theorem]{Proposition}
\newtheorem{lemma}[theorem]{Lemma}
\newtheorem{corollary}[theorem]{Corollary}

\theoremstyle{remark}
\newtheorem{remark}[theorem]{Remark}
\theoremstyle{remark}

\theoremstyle{definition}
\newtheorem{definition}[theorem]{Definition}

\allowdisplaybreaks[3]

\newcommand\define[1]{#1}

\newcommand\cO{\mathcal{O}}

\DeclareMathOperator\Hom{Hom}
\DeclareMathOperator\dHom{Hom}
\newcommand\dHomk{\dHom^\bullet}

\providecommand{\cE}{\mathcal{E}}
\providecommand{\cF}{\mathcal{F}}
\providecommand{\cG}{\mathcal{G}}

\providecommand{\cK}{\mathcal{K}}
\providecommand{\cM}{\mathcal{M}}

\providecommand{\cO}{\mathcal{O}}
\providecommand{\R}{\mathbb{R}}
\providecommand{\PP}{\mathbb{P}}
\providecommand{\Z}{\mathbb{Z}}

\providecommand{\C}{\mathbb{C}}

\providecommand{\F}{\mathbb{F}}
\providecommand{\Br}{\mathrm{Br}}
\providecommand{\Aut}{\mathrm{Aut}}
\providecommand{\AT}{\mathrm{AT}}

\providecommand{\Spec}{\operatorname{Spec}}
\providecommand{\Hilb}{\operatorname{Hilb}}
\providecommand{\SL}{\operatorname{SL}}
\providecommand{\diag}{\operatorname{diag}}

\newcommand\sm\shortminus

\newcommand{\hexsetup}{
	\def\a{1}
	\def\hs{0.8660254037844386} 
}

\newcommand{\hexagon}[3]{%
	\begin{scope}[shift={(#1,#2)}]
		\draw[thick]
		(0,\a) -- (\hs,\a/2) -- (\hs,-\a/2) -- (0,-\a) -- (-\hs,-\a/2) -- (-\hs,\a/2) -- cycle;
		\node at (0,0) {$#3$};
	\end{scope}
}

\title[Some faithful algebraic braid twist group actions for 3-fold crepant resolutions]{Some faithful algebraic braid group twist actions\\for 3-fold crepant resolutions}
\author{Luyu Zheng}
\email{zhengluyu84@gmail.com}
\begin{document}
\begin{sloppypar}
\phantom{.}\vspace{-1cm} 

\thispagestyle{empty}

\begin{abstract}
	Let $X(1,3,a)$ be a crepant resolution of the quotient singularity $\C^3/G$, where $G\subset \SL(3,\C)$ is a diagonal cyclic subgroup acting on $\C^3$ with weights $(1,3,a)$. 
	For each such $X(1,3,a)$, we construct a $(Q,W)$-configuration of spherical objects in the bounded derived category of coherent sheaves.
	When $a=9$, the derived category $D(X(1,3,9))$ admits a faithful algebraic braid twist group action of type~D induced by the associated $(Q,W)$-configuration.
	When $a=13$, the derived category $D(X(1,3,13))$ admits a faithful algebraic braid twist group action of type~E.
	These two cases illustrate the emergence of type~D and type~E patterns from specific geometric data, supporting a broader conjectural framework.
\end{abstract}
\maketitle


\setcounter{tocdepth}{1}
\tableofcontents

\section{Introduction}
Crepant resolutions of singular varieties play an important role in algebraic geometry. 
Many symmetries of such resolutions can be studied through the autoequivalences of the derived category of coherent sheaves, providing a useful way to understand the structure of the underlying singularity.

Spherical twists along spherical objects form a fundamental geometric source of such autoequivalences. Suitable collections of spherical objects give rise to braid group actions on derived categories. For example, Seidel–Thomas~\cite{ST} constructed a faithful Artin braid group action, and Nordskova–Volkov~\cite{NV} proved the existence of faithful braid group actions associated with simply-laced Dynkin types in an enhanced triangulated setting. In dimension two, Bridgeland~\cite{Brid} constructed type~ADE-configurations in the minimal resolutions of Kleinian singularities, and Brav–Thomas~\cite{BT} showed the faithfulness.

In dimension three, a type~A-configuration was constructed by Seidel–Thomas in~\cite[end of Section~3]{ST}. 
However, for crepant resolutions of threefold quotient singularities, it is much less clear how configurations related to Dynkin types~D and~E should appear, or how the corresponding braid group actions can be realized geometrically.

Qiu–Woolf~\cite{QW} constructed the Dynkin $Q$-configuration in the special derived category $D(\Gamma_N Q)$ for $\Gamma_N Q$ the Calabi–Yau-$N$ Ginzburg algebra associated with a Dynkin quiver $Q$ and showed that the induced $\Br(Q)$-action is faithful. We focus primarily on $G\text{-}\Hilb(\C^3)$, a crepant resolution of $\C^3/G$, where $G$ is a finite abelian subgroup of $\SL(3,\C)$. There is an equivalence $D(G\text{-}\Hilb(\C^3))\cong D(\Gamma_3 Q_{\operatorname{McK}})$~\cite[Corollary~4.4.8(2)]{Ginz}, where $Q_{\operatorname{McK}}$ is the McKay quiver of $G$. However, this equivalence does not directly yield Dynkin type configurations, since $Q_{\operatorname{McK}}$ is not of Dynkin type for most subgroups $G$.

In joint work with Donovan, we previously studied a specific example of a crepant resolution $X=G\text{-}\Hilb(\C^3)$ of a cyclic quotient singularity. We showed that the derived category $D(X)$ carries a faithful action of a quiver braid group, where the relevant quiver is a 3-cycle. The proof of faithfulness in that case relied on the type~A framework of Seidel–Thomas.

In this paper, we continue this study by analyzing two specific examples of crepant resolutions $G\text{-}\Hilb(\C^3)$ of multiply cyclic quotient singularities, with the aim of shedding light on the general case. In contrast to our joint work, we demonstrate faithful actions of the algebraic braid twist group $\AT(Q,W)$~\cite{GM,Qiu} via the $(Q,W)$-configurations defined in Definition~\ref{def.con}. Here $(Q,W)$ is a quiver with potential that can be transformed into quivers of type~D and~E, connecting our work to the framework of Nordskova–Volkov~\cite{NV}. This supports the view that an ADE-type pattern governs braid group actions arising from threefold singularities, with our $(Q,W)$-configurations serving as generalized, transformable versions of the classical Dynkin quivers.

\subsection{Conventions}

We take $X$ to be a smooth complex quasi-projective variety. We write $D(X)$ for the bounded derived category of coherent sheaves on~$X$, and take functors to be derived.
For a cyclic subgroup $G=\mu_r\subset\SL(3,\C)$ with weights $(1,s,r-s-1)$, we denote $G\text{-}\Hilb(\C^3)$ by $X(1,s,r-s-1)$. By~\cite{Nakamura}, $X(1,s,r-s-1)\to \mathbb{C}^3/G$ is a crepant resolution, and the junior simplex of $X(1,s,r-s-1)$ is described in~\cite{Craw,CR,IR}. The point $(0,0,0)\in \mathbb{C}^3/\mu_r$ is an isolated singularity if and only if $\operatorname{gcd}(s,r,r-s-1)=1$. Thus we only consider this case.

\subsection{Main theorem}\label{sec.main}
\begin{theorem}[{Theorem~\ref{thm.D6}}]
	The collection of spherical twists associated with exceptional surfaces in $X_1\coloneqq X(1,3,9)$ induces a faithful $\AT(Q_1,W_1)$-action on $D(X_1)$, where $\AT(Q_1,W_1)$ is the algebraic braid twist group associated with the quiver $(Q_1,W_1)$ in Figure~\ref{fig.quiver139(1)}.
\end{theorem}
Similarly,
\begin{theorem}[{Theorem~\ref{thm.E8}}]
	The collection of spherical twists associated with exceptional surfaces in $X_2\coloneqq X(1,3,13)$ induces a faithful $\AT(Q_2,W_2)$-action on $D(X_2)$, where $\AT(Q_2,W_2)$ is the algebraic braid twist group associated with the quiver $(Q_2,W_2)$ in Figure~\ref{fig.quiver1313(1)}.
\end{theorem}
\begin{figure}[htbp]
	\centering
	\begin{minipage}[b]{0.48\textwidth}
		\centering
		\begin{tikzcd}
			&& 1 \\
			& 2 && 6 \\
			3 && 4 && 5
			\arrow["a_1"{description}, from=1-3, to=2-2]
			\arrow["b_1"{description}, from=2-2, to=2-4]
			\arrow["c_2"{description}, from=2-2, to=3-1]
			\arrow["c_1"{description}, from=2-4, to=1-3]
			\arrow["b_3"{description}, from=2-4, to=3-3]
			\arrow["a_2"{description}, from=3-1, to=3-3]
			\arrow["b_2"{description}, from=3-3, to=2-2]
			\arrow["c_3"{description}, from=3-3, to=3-5]
			\arrow["a_3"{description}, from=3-5, to=2-4]
		\end{tikzcd}
		\caption{$(Q_1, W_1)$ associated with $X(1,3,9)$.}
		\label{fig.quiver139(1)}
	\end{minipage}
	\hfill
	\begin{minipage}[b]{0.48\textwidth}
		\centering
		\begin{tikzcd}
			&& 1 && 8 \\
			& 2 && 6 && 7 \\
			3 && 4 && 5
			\arrow["{a_1}"{description}, from=1-3, to=2-2]
			\arrow[from=1-5, to=1-3]
			\arrow["{b_1}"{description}, from=2-2, to=2-4]
			\arrow["{c_2}"{description}, from=2-2, to=3-1]
			\arrow["{c_1}"{description}, from=2-4, to=1-3]
			\arrow["{f_1}"{description}, from=2-4, to=2-6]
			\arrow["{b_3}"{description}, from=2-4, to=3-3]
			\arrow["{f_2}"{description}, from=2-6, to=3-5]
			\arrow["{a_2}"{description}, from=3-1, to=3-3]
			\arrow["{b_2}"{description}, from=3-3, to=2-2]
			\arrow["{c_3}"{description}, from=3-3, to=3-5]
			\arrow["{a_3}"{description}, from=3-5, to=2-4]
		\end{tikzcd}
		\caption{$(Q_2, W_2)$ associated with $X(1,3,13)$.}
		\label{fig.quiver1313(1)}
	\end{minipage}
\end{figure}

We sketch the idea of the proof. Generalizing the $A_n$-configuration from~\cite{ST}, we define a $(Q,W)$-configuration in $D(X)$ with an additional condition:
\begin{itemize}
	\item $T_{\cE_k}\cE_l \in \cE_t^\perp$ if there is a cycle $k \to l \to t \to k$ in $W$, where $T_{\cE_k}$ is the spherical twist associated with $\cE_k$ and $\cE_t^\perp \coloneqq {\mathcal{A} \in D(X) \mid \Hom^\bullet(\cE_t, \mathcal{A}) = 0}$.
\end{itemize}
We show that a $(Q,W)$-configuration induces an $\AT(Q,W)$-action on $D(X)$ when $X$ is Calabi–Yau, using the same method as in~\cite{DZ}.

To prove the existence of a $(Q_i,W_i)$-configuration in $D(X_i)$ for $i=1,2$, we first show that the push-forwards of certain line bundles on the exceptional surfaces $S_k \subset X_i$ are spherical. We then demonstrate in detail that the collection 
\[\{\cE_k\mid \; S_k\text{ is an irreducible exceptional surface in }X_i\}\] 
that we choose forms a $(Q_i,W_i)$-configuration and induces an $\AT(Q,W)$-action on $D(X_i)$. The methods we use to prove $T_{\cE_k}\cE_l \in \cE_t^\perp$ for a cycle $k \to l \to t \to k$ in $W_i$ differ from those in~\cite{DZ}, because not all exceptional surface pairs $(S_k, S_l)$ intersect in a fiber of $S_k$ or $S_l$; some intersections are curves with non-zero self-intersection number, requiring a more detailed analysis.

Finally, we show that the $(Q_1,W_1)$-configuration can be transformed into a $(D_6,0)$-configuration in $D(X_1)$ via elements of $\operatorname{AT}(Q_1,W_1)$. Faithfulness then follows from~\cite[Theorem~1]{NV}, and the transformation induces a group isomorphism $\operatorname{AT}(Q_1,W_1) \cong \operatorname{Br}(D_6)$. Similarly, we show that the $(Q_2,W_2)$-configuration can be transformed into an $(E_8,0)$-configuration, yielding faithfulness via~\cite[Theorem~1]{NV} and a group isomorphism $\operatorname{AT}(Q_2,W_2) \cong \operatorname{Br}(E_8)$.

\subsection{Contents of the paper}

In Section~\ref{sec.sph}, we review spherical objects in $D(X)$, their associated twists, and algebraic braid twist groups. We define $(Q,W)$-configurations and show they induce algebraic braid twist group actions on Calabi–Yau varieties.
In Section~\ref{sec.cal}, we compute $\Hom$ spaces associated with Cartier divisors and present detailed results for Calabi–Yau 3-folds, which are useful for later proofs.
In Section~\ref{sec.toric}, we briefly recall toric varieties via fans and orbits, construct toric line bundles from Cartier divisors, and describe Hirzebruch surfaces as toric varieties.
In Section~\ref{sec.exc}, we explain the exceptional surfaces in $X(1,s,r-s-1)$ and their intersections, illustrating with the example $X(1,3,9)$.
In Section~\ref{sec.139}, we prove our first main theorem for $X(1,3,9)$.
In Section~\ref{sec.1313}, we first describe the exceptional surfaces and their intersections in $X(1,3,13)$ and then prove our second main theorem.
Together, these sections construct the required $(Q,W)$-configurations, verify their properties, and establish the faithful braid group actions, thereby proving our main theorems.

\section{Spherical objects and braid groups}\label{sec.sph}
We recall the theory of spherical twists from~\cite{ST} for a smooth projective variety $X$ of dimension~$d$.

	\subsection{Spherical objects}
	
	\begin{definition}\label{definition.sph} An object $\mathcal{E}\in D(X)$ is \emph{spherical} if 
		\begin{itemize}
			\item 
			$\dHomk(\mathcal{E},\mathcal{E})=\mathbb{C}\oplus\mathbb{C}[-d]$
			and 
			\item $\mathcal{E}\otimes\omega_X\cong\mathcal{E}$
		\end{itemize}
		where $\dHomk(\mathcal{A},\mathcal{B})\coloneqq \bigoplus_{r\in\Z}\Hom_{D(X)}(\mathcal{A},\mathcal{B}[r])\cong \Gamma_X^\bullet(\mathcal{H}om_X(\mathcal{A},\mathcal{B}))$.
		\end{definition}
	
	For any $\mathcal{F} \in D(X)$, one can define the \emph{spherical twist} $T_{\mathcal{F}}\colon D(X)\to D(X)$ as in~\cite[Definition~8.3]{Huy}. Moreover, by~\cite[Exercise~8.5]{Huy}, for any spherical object $\mathcal{E}$, the action of the spherical twist on objects can be expressed as
	\[
	T_{\mathcal{E}}(\mathcal{A}) \cong
	\operatorname{Cone}\Bigl(
	\mathcal{E} \otimes \dHomk(\mathcal{E}, \mathcal{A})
	\longrightarrow \mathcal{A}
	\Bigr),
	\]
	where the morphism is the natural evaluation map.  
	If $\mathcal{E}$ is spherical, then $T_{\mathcal{E}}$ is an exact autoequivalence of $D(X)$~\cite[Proposition~8.6]{Huy}.
	
	\begin{proposition}\label{prop.dimandrelation}
		For two non-isomorphic spherical objects $\mathcal{E},\mathcal{F}\in D(X)$, let 
		$h=\dim_{\mathbb{C}}\dHomk(\mathcal{E},\mathcal{F})$, up to shifts. Then the following hold:
		\begin{itemize}
			\item If $h=0$, then $T_{\mathcal{E}}T_{\mathcal{F}}\cong T_{\mathcal{F}}T_{\mathcal{E}}$.
			\item If $h=1$, then  $T_{\mathcal{E}}T_{\mathcal{F}}T_{\mathcal{E}}
			\cong T_{\mathcal{F}}T_{\mathcal{E}}T_{\mathcal{F}}$ (the Artin braid relation).
			\item If $h\ge2$, then  
			$\langle T_{\mathcal{E}},T_{\mathcal{F}}\rangle$ forms a free subgroup of $\Aut\, D(X)$~\cite[Theorem~1.1]{Keat}.
		\end{itemize}
	\end{proposition}
	
	\begin{lemma}\label{lemma.sph}
		Assume $\cE$ is a spherical object in $D(X)$, and $\cF$ is a line bundle on $X$. Then 
		\begin{itemize}
			\item[(1)] $\cF\otimes\cE$ is also spherical in $D(X)$;
			\item[(2)] For any $\mathcal{A}\in D(X)$ and any integer $r$, $T_{\cE\otimes\cF[r]}(\mathcal{A}\otimes \cF)\cong T_\cE(\mathcal{A})\otimes \cF$.
		\end{itemize}
	\end{lemma}
	\begin{proof}
		By the definition of spherical object,  $\cF\otimes\cE\otimes\omega_{X}\cong\cF\otimes\cE$. Then (1) is proved by
		
		\[\dHomk(\cF\otimes\cE,\cF\otimes\cE)\cong\Gamma_X^\bullet(\mathcal{H}om_X(\cE,\cE)\otimes\cF^\vee\otimes\cF)\cong\dHomk(\cE,\cE)=\C\oplus\C[- d],\]
		where the first isomorphism uses~\cite[III, Proposition~6.7]{Har} and the last isomorphism uses $\cF^\vee\otimes\cF\cong\cO_X$.
		
		Since \(\dHomk(\cE\otimes\cF[r],\, \mathcal{A}\otimes\cF)=\dHomk(\cE,\, \mathcal{A})[-r]\), (2) is given by
		\begin{align*}
			T_{\cE\otimes\cF[r]}(\mathcal{A}\otimes\cF)
			&=\mathrm{Cone}\{\cE\otimes\cF[r]\otimes \dHomk(\cE\otimes\cF[r],\, \mathcal{A}\otimes\cF)\to\mathcal{A}\otimes\cF  \}\\
			&\cong \mathrm{Cone}\{\cE\otimes\cF\otimes \dHomk(\cE,\, \mathcal{A})\to\mathcal{A}\otimes\cF  \}\\
			&\cong \mathrm{Cone}\{\cE\otimes \dHomk(\cE,\, \mathcal{A})\to\mathcal{A} \}\otimes\cF\\
			&\cong T_{\cE}\mathcal{A}\otimes \cF
		\end{align*}
	\end{proof}
	
	Similar to the definition of $A_n$-configuration in~\cite{ST}, we can define a $(Q,W)$-configuration.
	\begin{definition}\label{def.con}
		Suppose that $(Q,W)$ is a quiver with potential such that there is at most one arrow between any two vertices and the length of any cycle in $W$ is $3$. An \emph{$(Q,W)$-configuration} in $D(X)$ is a collection of spherical objects 
		$\{\mathcal{E}_k\}_{k\in Q_0}$ satisfying the following conditions:
		\begin{itemize}
			\item $\dim_{\mathbb{C}}\dHomk(\mathcal{E}_i,\mathcal{E}_j)=1$, if there is an arrow between $i$ and $j$;
			\item $\dim_{\mathbb{C}}\dHomk(\mathcal{E}_i,\mathcal{E}_j)=0$, if there is no arrow between $i$ and $j$;
			\item $T_{\cE_k}\cE_l\in\cE_t^\perp$ if there is a cycle $k\to l\to t\to k$ in $W$, where  $T_{\cE_k}$ is the spherical twist associated with $\cE_k$ and $\cE_t^\perp \coloneqq  \{\mathcal{A}\in D(X)\mid \dHomk(\cE_t,\,\mathcal{A})=0\}$.
		\end{itemize}
	\end{definition}
	
	Denote $T_i\coloneqq T_{\cE_i}$  the spherical twist associated with the spherical object $\cE_i$.
	
	The orthogonality condition in Definition~\ref{def.con} is independent of the choice of vertex in a 3-cycle, as shown by the following lemma.
	\begin{lemma}\label{lemma.change}
		Assume $\{\cE_i,\cE_j\}$ is an $A_2$-configuration for any $1\leq i\neq j\leq3$ and $T_1\cE_2\in\cE_3^\perp$. Then $T_2\cE_3\in\cE_1^\perp$ and $T_3\cE_1\in\cE_2^\perp$
	\end{lemma}
	\begin{proof}
		Since $\{\cE_i,\cE_{i+1}\}$ is an $A_2$-configuration, by~\cite[Lemma~2.3(1)]{DZ}, there is an integer $r_i$ such that $T_i\cE_{i+1}\cong T_{i+1}^{-1}\cE_{i}[r_i]$ with indices taken modulo $3$, for $i=1,2,3$.
		Since $T_i\in \Aut\,D(X)$, we have 
		\[\dHomk(\cE_1,T_2\cE_3)\cong\dHomk(T_2^{-1}\cE_1,\cE_3)\cong\dHomk(T_1\cE_2,\cE_3)[r_1]=0; \]
		and 
		\[\dHomk(\cE_2,T_3\cE_1)\cong \dHomk(T_3^{-1}\cE_2,\cE_1)\cong\dHomk(T_2\cE_3,\cE_1)[r_2]=0. \]
		Thus, $T_2\cE_3\in\cE_1^\perp$ and $T_3\cE_1\in\cE_2^\perp$.
	\end{proof}
	
	When $(Q,W)=(A_n,0)$, the definition is compatible with the definition of $A_n$-configuration in~\cite[Definition~1.1(b)]{ST}. Moreover, an $A_n$-configuration induces a faithful action of the braid group $\Br_{n+1}$ on $D(X)$ via spherical twists $T_i\coloneqq T_{\cE_i}$~\cite[Theorem~1.3]{ST}, when $d\geq 2$.
	
	When $(Q,W)=(\Gamma,0)$, where $\Gamma$ is an ADE-type Dynkin quiver, the definition is compatible with the definition of $\Gamma$-configuration in~\cite[Definition~2.2]{NV}. Moreover, a $\Gamma$-configuration induces a faithful action of the braid group $\Br(\Gamma)$ on $D(X)$ via spherical twists $T_i\coloneqq T_{\cE_i}$~\cite[Theorem~1]{NV}, when $d\neq 1$.
	
	For a quiver with potential $(Q,W)$, Qiu~\cite{Qiu} and Grant-Marsh~\cite{GM} also define an algebraic braid twist group associated with $(Q,W)$, which is compatible with $\AT(A_n,0)=\Br(A_n)=\Br_{n+1}$.
	
	\begin{definition}[Definition~10.1\cite{Qiu}]
		Suppose that $(Q,W)$ is a quiver with potential such that there is at most one arrow between any two vertices. The algebraic braid twist group $\AT(Q,W)$ is defined by the following presentation
		\begin{itemize}
			\item generators $\beta_i$ for $i\in Q_0$;
			\item there is a relation $Br(\beta_i,\beta_j)$ if there is an arrow between $i,j$; otherwise there is a relation $Co(\beta_i,\beta_j)$;
			\item there are relations $R_i=R_j$ for any $i\neq j$, if there is a cycle \(Y\colon 1\to2\to\cdots\to m\to 1\) in $W$, where $R_i=\beta_i\beta_{i+1}\cdots\beta_{2m+i-3}$ indices are taken modulo $m$.
		\end{itemize}
	\end{definition}
	
	Similarly, a \emph{$(Q,W)$-configuration} in $D(X)$ induces an $\AT(Q,W)$-action on $D(X)$, when $X$ is Calabi-Yau.

	\subsection{Spherical objects in Calabi-Yau varieties}
	
	Now, we assume $X$ is Calabi-Yau, so that $\omega_{X}$ is trivial. 
	
	\begin{lemma}\label{lemma.equivalence}
		Assume $\Theta\in \Aut\, D(X)$ and $\cE$ is spherical in $D(X)$. Then 
		\begin{itemize}
			\item[(1)] $\Theta(\cE)$ is spherical in $D(X)$,
			\item[(2)] $T_{\Theta(\cE)}\cong \Theta T_{\cE} \Theta^{-1}$.
		\end{itemize}
	\end{lemma}
	\begin{proof}
		Due to $\Theta\in \Aut\, D(X)$, we have $\dHomk(\Theta(\cE),\Theta(\cE))\cong\dHomk(\cE,\cE)=\C\oplus\C[-d]$. Then (1) is given by Calabi-Yau $X$.
		
		In addition, $\dHomk(\Theta(\cE),\mathcal{A})\cong \dHomk(\cE,\Theta^{-1}(\mathcal{A}))$ induces the following commutative diagram between two exact triangles in $D(X)$
		\[
		\begin{tikzcd}
			{\Theta(\cE)\otimes\dHomk(\Theta(\cE),\mathcal{A})} && {\mathcal{A}} && {T_{\Theta(\cE)}(\mathcal{A})} \\
			{\Theta(\cE\otimes\dHomk(\cE,\Theta^{-1}(\mathcal{A})))} && {\Theta\circ\Theta^{-1}(\mathcal{A})} && {\Theta T_{\cE}(\Theta^{-1}(\mathcal{A})),}
			\arrow[from=1-1, to=1-3]
			\arrow["\cong"{description}, from=1-1, to=2-1]
			\arrow[from=1-3, to=1-5]
			\arrow["\cong"{description}, from=1-3, to=2-3]
			\arrow[from=2-1, to=2-3]
			\arrow[from=2-3, to=2-5]
		\end{tikzcd}
		\]
		Since $\Theta\circ\Theta^{-1}\cong \mathrm{id}$ in $\Aut\,D(X)$. This proves (2).
	\end{proof}
	
		\begin{proposition}\label{prop.cycle}
		Suppose that $(Q,W)$ is a quiver with potential such that there is at most one arrow between any two vertices and the length of any cycle in $W$ is $3$. Then, a \emph{$(Q,W)$-configuration} in $D(X)$ induces an $\AT(Q,W)$-action on $D(X)$ via $\beta_i\mapsto T_i$.
	\end{proposition}
	\begin{proof}
		Using Proposition~\ref{prop.dimandrelation}, it suffices to verify the cycle relations corresponding to 3-cycles in $W$ i.e., $T_1\cE_2\in\cE_3^\perp$ induces $T_1T_2T_3T_1\cong T_2T_3T_1T_2\cong T_3T_1T_2T_3$, if there is a cycle $1\to 2\to 3\to 1$ in $W$.
		
		By~\cite[Lemma~2.3(1)]{DZ}, there is an integer $r$ such that $T_1\cE_{2}\cong T_{2}^{-1}\cE_{1}[r]$. Thus $T_1\cE_2\in\cE_3^\perp$ induces $T_3(T_1\cE_2)\cong T_1\cE_2\cong T_2^{-1}\cE_1[r]$. Hence 
		\[(T_3T_1)T_2(T_3T_1)^{-1}\cong T_{(T_3T_1)\cE_2}\cong T_{T_2^{-1}\cE_1[r]}\cong T_{T_2^{-1}\cE_1}\cong T_2^{-1}T_1 T_2,\] 
		where, the first and last isomorphisms use Lemma~\ref{lemma.equivalence} and~\ref{lemma.sph}(2).
		
		In addition, by rotational symmetry, Lemma~\ref{lemma.change} induces $T_2T_3T_1T_2\cong T_3T_1T_2T_3$.
	\end{proof}
	
	\begin{lemma}\label{lemma.self}
		Assume $\{\cE_1,\cE_2\}$ is an $A_2$-configuration. Then $T_2T_1(\cE_2)\cong\cE_1[r]$ for some shift $r$.
	\end{lemma}
	\begin{proof}
		By~\cite[Lemma~2.3(1)]{DZ}, there is an integer $r$ such that $T_1\cE_2=T_2^{-1}\cE_1[r]$. Then we have
		\[T_2T_1(\cE_2)\cong T_2T_2^{-1}\cE_1[r]\cong\cE_1[r].\qedhere\]
	\end{proof}
	
	\begin{lemma}\label{lemma.A3}
		Assume $\{\cE_1,\cE_2,\cE_3\}$ is an $A_3$-configuration. Then $\{\cE_1,T_2(\cE_3),\cE_3\}$ is an $A_3$-configuration.
	\end{lemma}
	\begin{proof}
		 Using~\cite[Lemma~2.3(1)]{DZ}, there is an integer $r$ such that $T_2\cE_3=T_3^{-1}\cE_2[r]$. Then using $T_i\in \Aut\, D(X)$, we have
		\[\dHomk(\cE_1,T_2(\cE_3))\cong\dHomk(\cE_1,T_3^{-1}\cE_2)[r]\cong \dHomk(T_3\cE_1,\cE_2)[1-r]\cong\dHomk(\cE_1,\cE_2)[r]\]
		where the last isomorphism uses $\dHomk(\cE_1,\cE_3)=0$. 
		
		And
		\[\dHomk(T_2(\cE_3),\cE_3)\cong\dHomk(T_3^{-1}\cE_2,\cE_3)[-r]\cong \dHomk(\cE_2,T_3\cE_3)\cong\dHomk(\cE_2,\cE_3)[-r- d],\]
		where the last isomorphism uses $T_3\cE_3=\cE_3[- d]$.
		
		Thus, \(\dim \dHomk(\cE_1,T_2(\cE_3))=1\) and \(\dim \dHomk(T_2(\cE_3),\cE_3)=1\). The claim then follows from Lemma~\ref{lemma.equivalence}(1) and Serre duality.
	\end{proof}
	
	Consider a special quiver with potential $(Q,W)$ as shown in Figure~\ref{fig.quiver139}
	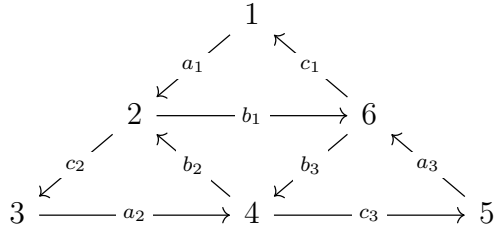
\begin{figure}[htbp]
		\centering
		\begin{tikzcd}
			&& 1 \\
			& 2 && 6 \\
			3 && 4 && 5
			\arrow["a_1"{description}, from=1-3, to=2-2]
			\arrow["b_1"{description}, from=2-2, to=2-4]
			\arrow["c_2"{description}, from=2-2, to=3-1]
			\arrow["c_1"{description}, from=2-4, to=1-3]
			\arrow["b_3"{description}, from=2-4, to=3-3]
			\arrow["a_2"{description}, from=3-1, to=3-3]
			\arrow["b_2"{description}, from=3-3, to=2-2]
			\arrow["c_3"{description}, from=3-3, to=3-5]
			\arrow["a_3"{description}, from=3-5, to=2-4]
		\end{tikzcd}
		\caption{$(Q, W=\sum_1^3(-a_kb_kc_k)+b_3b_2b_1)$}
		\label{fig.quiver139}
	\end{figure}
	
	When $X$ is Calabi-Yau, a $(Q,W)$-configuration in $D(X)$ induces a $D_6$-configuration in $D(X)$.
	\begin{proposition}\label{prop.D6}
		Assume $X$ is Calabi-Yau and \(\{\cE_1,\cE_2,\cE_3,\cE_4,\cE_5,\cE_6\}\) is a $(Q,W)$-configuration in $D(X)$. Then \[\cF_1\coloneqq\cE_5,\quad\cF_2\coloneqq T_4T_5T_2(\cE_6),\quad\cF_3\coloneqq\cE_4,\quad\cF_4\coloneqq\cE_3,\quad\cF_5\coloneqq T_2(\cE_3),\quad \cF_6\coloneqq \cE_1\] generate a $D_6$-configuration as shown in Figure~\ref{fig.D6con}.
		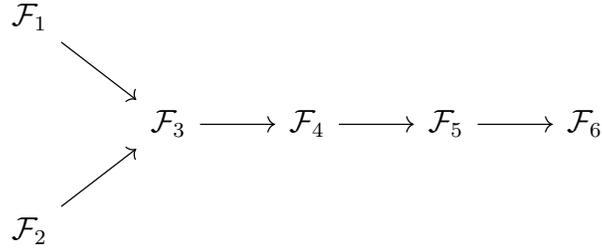
\begin{figure}[htbp]
			\centering
			\begin{tikzcd}
				\cF_1 \\
				& \cF_3 & \cF_4 & \cF_5 & \cF_6\\
				\cF_2
				\arrow[ from=1-1, to=2-2]
				\arrow[ from=3-1, to=2-2]
				\arrow[ from=2-2, to=2-3]
				\arrow[ from=2-3, to=2-4]
				\arrow[ from=2-4, to=2-5]
			\end{tikzcd}
			\caption{$D_6$-configuration}
			\label{fig.D6con}
		\end{figure}
	\end{proposition}
	\begin{proof}
		Using Lemma~\ref{lemma.self} and~\ref{lemma.A3}, we have an $A_5$-subconfiguration
		\(\{\cF_1\coloneqq\cE_5,\cF_3\coloneqq\cE_4,\cF_4\coloneqq\cE_3,\cF_5\coloneqq T_2(\cE_3), \cF_6\coloneqq \cE_1\}\) since there is a cycle $2\to3\to 4\to 2$ in $W$. It suffices to show that the $\dim\dHomk(\cF_2,\cF_3)=1$ and $\dim\dHomk(\cF_2,\cF_k)=0$ for $k\neq 2,3$.
		
		By~\cite[Lemma~2.3(1)]{DZ}, there is an integer $r_i$ such that $T_i\cE_{i+1}\cong T_{i+1}^{-1}\cE_i[r_i]$ with indices taken modulo $6$, for $1\leq i\leq 6$. Then using $T_i\in\Aut\, D(X)$, we have 
		\begin{itemize}
			\item
			\begin{align}
				\dHomk(\cF_2,\cF_1)=&\dHomk(T_4T_5T_2(\cE_6),\cE_5)\notag\\
				&\cong\dHomk(T_2(\cE_6),(T_4T_5)^{-1}\cE_5)\notag\\
				&\cong\dHomk(T_2(\cE_6),\cE_4)[-r_4]\label{sim.lemmaself}\\
				&=0\label{sim.264cylce},
			\end{align}
			where \eqref{sim.lemmaself} uses Lemma~\ref{lemma.self} and \eqref{sim.264cylce} is because there is a cycle $2\to 6\to 4\to 2$ and using Serre duality.
			
			\item 
			\begin{align}
				\dHomk(\cF_2,\cF_3)=&\dHomk(T_4T_5T_2(\cE_6),\cE_4)\notag\\
				&\cong\dHomk(T_2(\cE_6),T_5^{-1}\cE_4)[ d]\label{sim.self}\\
				&\cong\dHomk(T_2(\cE_6),T_4\cE_5)[-r_4+ d]\notag\\
				&\cong\dHomk(T_2(\cE_6),\cE_5)[-r_4+ d]\label{sim.246cycle}\\
				&\cong\dHomk(\cE_6,\cE_5)[-r_4+ d]\label{sim.E5inE2perp},
			\end{align}
			where \eqref{sim.self} uses $T_4^{-1}\cE_4\cong\cE_4[ d]$, \eqref{sim.246cycle} is because there is a cycle $ 2\to 6\to 4\to 2$ in $W$ and \eqref{sim.E5inE2perp} is because there is no arrow between $2$ and $5$ in $Q$.
			
			\item 
			\begin{align}
				\dHomk(\cF_2,\cF_4)=&\dHomk(T_4T_5T_2(\cE_6),\cE_3)\notag\\
				&\cong\dHomk(T_2T_5(\cE_6),T_3\cE_4)[-r_3]\label{sim.E_2}\\
				&\cong\dHomk(T_5(\cE_6),T_3\cE_4)[-r_3]\label{sim.234cycle}\\
				&\cong \dHomk(T_5(\cE_6),\cE_4)[-r_3]\label{sim.E_3perp}\\
				&=0\label{sim.456cycle},
			\end{align}
			where \eqref{sim.E_2} is because there is no arrow between $2$ and $5$ in $Q$, \eqref{sim.234cycle} is because there is a cycle $2\to 3\to 4\to 2$ in $W$, \eqref{sim.E_3perp} is because there is no arrow between $3$ and $k$ in $Q$, for $k=5,6$, and \eqref{sim.456cycle} is because there is a cycle $5\to 6\to 4\to 5$ in $W$ and using Serre duality.
			
			\item
			\begin{align}
				\dHomk(\cF_2,\cF_5)&=\dHomk(T_4T_5T_2(\cE_6),T_2(\cE_3))\notag\\
				&\cong\dHomk(T_5T_2(\cE_6),T_2(\cE_3))\label{sim.234cycle(1)}\\
				&\cong \dHomk(T_5(\cE_6),(\cE_3))\label{sim.T_2}\\
				&=0\label{sim.563}
			\end{align}
			where \eqref{sim.234cycle(1)} is because there is a cycle $2\to 3\to 4\to 2$, \eqref{sim.T_2} follows that there is no arrow between $2$ and $5$ in $Q$, and \eqref{sim.563} is because there is no arrow between $3$ and $k$ in $Q$, for $k=5,6$.
			
			\item 
			\[\dHomk(\cF_2,\cF_6)=\dHomk(T_4T_5T_2(\cE_6),\cE_1)\cong\dHomk(T_2(\cE_6),\cE_1)=0,\]
			where the isomorphism is because there is no arrow between $1$ and $k$ in $Q$, for $k=4,5$ and the final equality is because there is a cycle $2\to 6\to 1\to 2$ in $W$ and using Serre duality.
			
			\end{itemize}
		Thus, \[\dim\dHomk(\cF_2,\cF_k)=
		\begin{cases}
			1, & k=3\\
			0, & k=1,4,5,6,
		\end{cases}
		\]
		and the claim is given by Serre duality.
	\end{proof}

\section{Homological calculations for spherical sheaves}

	\subsection{General calculations}\label{sec.cal}
	The following results will be used to compute the $\Hom$ spaces between spherical objects later.
	
	Let $D_1$ and $D_2$ be two Cartier divisors in $X$ with embeddings $i_k\colon D_k\to {X}$. Assume their intersection $C\coloneqq D_1\cap D_2$ is also Cartier in each $D_k$, with embeddings $j_k\colon C\to D_k$.
	\begin{equation*}
		\begin{tikzcd}
			& {X}                                   &                    \\
			{D_1} \arrow[ru, hook,"i_1"] &                                     & {D_2} \arrow[lu, hook', "i_2"'] \\
			& C \arrow[lu, hook, "j_1"] \arrow[ru, hook', "j_2"'] &                   
		\end{tikzcd}
	\end{equation*}
	\begin{proposition}\label{prop.otimes}
		Assume $\cE$ is locally free over $X$ and $\cF\in D(D_2)$. Then
		\begin{equation*}
			\mathcal{H}om_X(i_{1*}\cO_{D_1}\otimes\cE,i_{2*}\cF)\cong i_{2*} j_{2*}(\cO_{C}(D_1)\otimes j_2^*\cF\otimes j_2^*i_2^*\cE^\vee)[-1].
		\end{equation*}
	\end{proposition}
	\begin{proof}
		Since $\cE$ is locally free, using~\cite[III~Proposition~6.7]{Har} and~\cite[Proposition~2.5]{DZ},
		\begin{align*}
			\mathcal{H}om_X(i_{1*}\cO_{D_1}\otimes\cE, i_{2*}\cF)&\cong \mathcal{H}om_X(i_{1*}\cO_{D_1}, i_{2*}\cF)\otimes\cE^\vee\\
			&\cong i_{2*} j_{2*}(\cO_{C}(D_1)\otimes j_2^*\cF\otimes j_2^*i_2^*\cE^\vee)[-1].
		\end{align*}
	\end{proof}

	Assume $D_k'$ is a divisor in $X$ such that $C_k\coloneqq D_k\cap D_k'$ is a divisor in $D_k$, for $k=1,2$. Since $i_{k}$ are closed embeddings, $i_{k*}\cO_{D_k}\otimes\cO_X(-D_k')\cong i_{k*}\cO_{D_k}(-C_k)$, for $k=1,2$.
	
	\begin{corollary}\label{coro.Homspace}
		$\mathcal{H}om_X(i_{1*}\cO_{D_1}(-C_1), i_{2*}\cO_{D_2}(-C_2))$ is isomorphic to $i_{2*} j_{2*}(\cO_{C}(D_1+D_1'-D_2'))[-1]$.
	\end{corollary}
	\begin{proof}
		Since $\cO_X(-D_k')$ is locally free, using Proposition~\ref{prop.otimes}, we have
		\begin{align*}
			\mathcal{H}om_X(i_{1*}\cO_{D_1}(-C_1), i_{2*}\cO_{D_2}(-C_2))&\cong \mathcal{H}om_X(i_{1*}\cO_{D_1}\otimes\cO_X(-D_1'), i_{2*}\cO_{D_2}(-C_2))\\
			&\cong  i_{2*} j_{2*}(\cO_{C}(D_1)\otimes j_2^*\cO_{D_2}(-C_2)\otimes j_2^*i_2^*\cO_X(D_1'))[-1]\\
			&\cong  i_{2*}j_{2*}\cO_{C}(D_1-D_2'+D_1')[-1]
		\end{align*}
	\end{proof}
	
	\begin{lemma}[\hspace*{-0.4em}{\cite[Lemma~2.6]{DZ}}]\label{lemma.number}
		Assume $\dim X=3$, $C\cong\PP^1$ and all divisors are smooth. Assume $C^2=c_2$ in $D_2$ is the self-intersection of $C$ in $D_2$. Then $\mathcal{O}_{C}(D_1)\cong\mathcal{O}_{C}(c_2)$.
	\end{lemma}
	
	\begin{corollary}\label{coro.Hom}
		Assume $\dim X=3$, $C\cong\PP^1$ and $D_1'-D_2'\sim -D_1+D_3'$ as divisors in $X$ and $D_3'\cap D_2=\varnothing$. Then $\dHomk(i_{1*}\cO_{D_1}(-C_1), i_{2*}\cO_{D_2}(-C_2))=\C[-1]$
	\end{corollary}
	\begin{proof}
		Using Corollary~\ref{coro.Homspace} and the assumption, we have
		\[\dHomk(i_{1*}\cO_{D_1}(-C_1), i_{2*}\cO_{D_2}(-C_2))\cong\Gamma_{C}^\bullet\cO_C(D_1+D_1'-D_2')[-1]\cong \Gamma_{C}^\bullet\cO_C[-1].\]
		Since $C\cong \PP^1$, the argument is proved.
	\end{proof}
	
	\subsection{Spherical objects on Calabi-Yau 3-folds}
	
		Assume now that $X$ is Calabi-Yau, so that $\omega_{X}$ is trivial. Let $D$ be a complete, connected, rational hypersurface in $X$, with embedding $i\colon D\to X$. 
	
	\begin{proposition}[\hspace*{-0.4em}{\cite[Proposition~5.1]{DZ}}]\label{prop.sph}
		$i_*\mathcal{O}_{D}$ is a spherical object in $D(X)$.
	\end{proposition}
	\begin{proof}
		Since $D$ is rational and connected, $\mathcal{O}_{D}$ is an exceptional object in $D(D)$. $X$ is a smooth quasi-projective variety and $i\colon D\to X$ is an embedding of a complete connected hypersurface. Since $X$ is Calabi-Yau, $i^* \omega_{X}$ is trivial. Thus, by \cite[Proposition~3.15]{ST}, $i_*\mathcal{O}_{D}$ is spherical in $D(X)$. 
	\end{proof}
	
	\begin{remark}
		In this paper, whenever we refer to a spherical object $\cE\in D(X)$, we always assume it is of the form $\cE\cong i_*\cG$ for some $\cG\in D(D)$. For spherical objects $\cF,\cE\in D(X)$, Serre duality implies (see \cite[before Proposition~3.15]{ST}) that 
		\[\dHomk(\cF,\cE)=\dHomk(\cE,\cF)^\vee[- d],\]
		and in particular the spherical twist $T_\cE$ is also an autoequivalence of $D(X)$.
	\end{remark}
	
	From now on, we assume that $X$ is a smooth Calabi-Yau 3-fold and that $C$ is isomorphic to $\PP^1$. We also assume that all Cartier divisors $D_k$ and $D_k'$ defined in Section~\ref{sec.cal} are complete, connected, and rational. Then by Proposition~\ref{prop.sph} and Lemma~\ref{lemma.sph}(1), $\cE_k\coloneqq i_{k*}\cO_{D_k}(-C_k)$ is a spherical object in $D(X)$ for $k=1,2$. 
	
	\begin{lemma}\label{lemma.sum}
		Assume $\dim X=3$ and $C\cong \PP^1$. Let $D_k$ be smooth divisors and let $c_k$ denote the self-intersection number of $C$ in $D_k$, for $k=1,2$. Then 
		\[
		c_1 + c_2 = -2.
		\]
	\end{lemma}
	
	\begin{proof}
		By Corollary~\ref{coro.Homspace} and Lemma~\ref{lemma.number}, we have
		\[
		\mathcal{H}om_X(i_{2*}\cO_{D_2}(-C),\, i_{1*}\cO_{D_1}) \cong i_{1*} j_{1*} \cO_C(D_1 + D_2)[-1] \cong i_{1*} j_{1*} \cO_C(c_1 + c_2)[-1].
		\]
		
		Taking global sections gives
		\[
		\dHomk(i_{2*}\cO_{D_2}(-C), i_{1*}\cO_{D_1}) \cong \Gamma_C^\bullet \cO_C(c_1+c_2)[-1].
		\]
		
		By Serre duality and Corollary~\ref{coro.Hom}, this is isomorphic to
		\[
		\dHomk(i_{1*}\cO_{D_1}, i_{2*}\cO_{D_2}(-C))^\vee[-3] \cong \C[-2].
		\]
		
		Therefore $\Gamma_C^\bullet \cO_C(c_1+c_2) \cong \C[-1]$. 
		Since $C \cong \PP^1$, this implies $\cO_C(c_1+c_2) \cong \cO_{\PP^1}(-2)$, hence $c_1 + c_2 = -2$.
	\end{proof}

	\begin{lemma}\label{lemma.extcon}
		Assume $D_2'\sim D_1$ and $D_1'\cap D_2=\varnothing$ and $C\cong\PP^1$.
		Let $D_{12}\coloneqq D_1\cup D_2$ and $i_{12}\colon D_{12}\to X$ be the closed embedding.
		Then $T_{\cE_1}\cE_2\cong i_{12*}\cO_{D_{12}}(-D_1')$.
	\end{lemma}
	\begin{proof}
		Since $D_1$ and $D_2$ are smooth, then we have the following short exact sequence
		\begin{equation}\label{ext.con1}
			0\longrightarrow\mathcal{I}_{D_{2}|C}\longrightarrow\cO_{D_{12}}\longrightarrow\cO_{D_{1}}\longrightarrow0,
		\end{equation}
		on $D_{12}$. The extension is non-split since $D_{1}\cap D_{2}$ is non-empty. In addition, $\mathcal{I}_{D_{2}|C}\cong \cO_{D_2}(-C)$ since $D_2$ and $C$ are smooth and $C$ is a Cartier divisor in $D_2$.
		
		Applying $i_{12*}$ to \eqref{ext.con1}, we get the exact triangle in $D(X)$
		\begin{equation}\label{ext.con2}
			i_{2*}\cO_{D_2}(-C)\longrightarrow i_{12*}\cO_{D_{12}}\longrightarrow i_{1*}\cO_{D_{1}},
		\end{equation}
		Applying $\otimes\cO_{X}(-D_1')$ to \eqref{ext.con2}, we get the exact triangle
		\begin{equation}\label{ext.con3}
			 i_{2*}\cO_{D_2}(-C)\longrightarrow i_{12*}\cO_{D_{12}}(-D_1')\longrightarrow i_{1*}\cO_{D_{1}}(-C_1),
		\end{equation}
		since $D_2\cap D_1'=\varnothing$.
		 
		By Corollary~\ref{coro.Hom} and \cite[Lemma~2.3(2)]{DZ}, the spherical twist $T_{\cE_1} \cE_2$ fits into the unique non-split extension
		\begin{equation}\label{ext.con4}
			 \cE_2\longrightarrow T_{\cE_1}\cE_2\longrightarrow \cE_1,
		\end{equation} 
		
		 Combining \eqref{ext.con3} and \eqref{ext.con4}, the claim is shown, since $\cE_2\cong i_{2*}\cO_{D_2}\otimes \cO_X(-D_2')\cong i_{2*}\cO_{D_2}\otimes \cO_X(-D_1)$, where $D_2'\sim D_1$.
	\end{proof}

	Consider the smooth Cartier divisor $D_3$ in $X$ intersects with $D_1$ and $D_2$ as shown in Figure~\ref{fig.rel1}.
	\begin{figure}[htbp]
		\centering
		\begin{tikzpicture}[scale=2,semithick,
			dot/.style = {circle, fill=black, inner sep=1pt}]
			\coordinate (O) at (0,0);
			\coordinate (V2) at (90:1); 
			\coordinate (V3) at (210:1); 
			\coordinate (V1) at (330:1);
			\coordinate (S2) at (-90:1); 
			\coordinate (S1) at (150:1); 
			\coordinate (S3) at (30:1);
			\coordinate (V21) at ($ (V3) !1/2! (S2) $);
			\coordinate (V22) at ($ (O) !1/2! (V1) $);

			\draw (S3) -- (V2) -- (S1) -- (V3) -- (S2) -- (V1) -- cycle; 
			\draw (V1) -- (O);
			\draw (V2) -- (O);
			\draw (V3) -- (O);

			\node[label={180:${D}_1$}] at (S1) {};
			\node[label={270:${D}_2$}] at (S2) {};
			\node[label={0:${D}_3$}] at (S3) {};
			\node[label={-30:$C_{23}$}] at (V1) {};
			\node[label={210:$C$}] at (V3) {};
			\node[label={90:$C_{13}$}] at (V2) {};
			\node[dot,label={270:$p$}] at (O) {};
		\end{tikzpicture}
		\caption{The relation among $D_1$, $D_2$ and $D_3$}
		\label{fig.rel1}
	\end{figure}
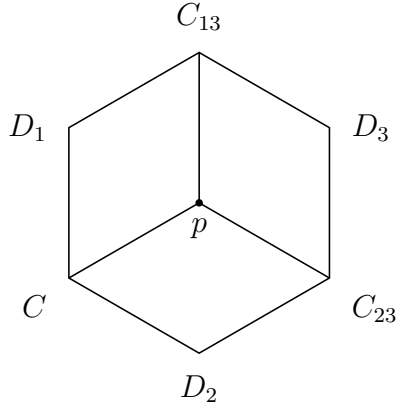 
	
Let $C_{k3} \coloneqq D_k \cap D_3$ for $k=1,2$. Denote $C_3=C_{13}\cup C_{23}$ and $j_3\colon  C_3\to D_{12}$. Let $\{p\}\coloneqq D_1\cap D_2\cap D_3$.
	
	\begin{lemma}\label{lemma.homexten}
		Assume
		\begin{itemize}
			\item $C_{k3}\cong\PP^1$ for $k=1,2$;
			\item $\dHomk(\cE_3,\cE_2)=\C[-2]$ and $\dHomk(\cE_3,\cE_1)=\C[-1]$;
			\item $T_{\cE_1}\cE_2\cong i_{12*}\cM$, where $\cM$ is a line bundle on $D_{12}$.
		\end{itemize}
		Then there is a line bundle $\cK\coloneqq\cO_{C_3}(D_3+D_3')\otimes j_3^*\cM$ on $C_3$ such that
		\begin{itemize}
			\item[(1)] \(\mathcal{H}om_X(i_{3*}\cO_{D_3}\otimes\cE_3,T_{\cE_1}\cE_2)\cong i_{12*}j_{3*}\cK[-1]\);
			\item[(2)] \(\cK|_{C_{13}}\cong  j_{31}^*\cM\otimes \cO_{C_{13}}(D_1')\);
			\item[(3)] \(\cK|_{C_{23}}\cong  j_{32}^*\cM\otimes\cO_{C_{23}}(D_2'-2)\).
		\end{itemize}
	\end{lemma}
	\begin{proof}
		 Using Corollary~\ref{coro.Homspace}, (1) is given by
		\[\mathcal{H}om_X(\cE_3,T_{\cE_1}\cE_2)\cong i_{12*}(j_{3*}\cO_{C_3}(D_3+D_3')\otimes\cM)[-1]\cong i_{12*}j_{3*}(\cK)[-1]\]
		
		Again by Corollary~\ref{coro.Homspace}, we have
		\[\mathcal{H}om_X(\cE_3,\cE_1)\cong i_{1*}j_{31*}\cO_{C_{13}}(D_3+D_3'-D_1')[-1]\]
		Since $\dHomk(\cE_3,\cE_1)=\C[-1]$ and $C_{13}\cong\PP^1$, then $\cO_{C_{13}}(D_3+D_3'-D_1')\cong \cO_{C_{13}}$ i.e., $\cO_{C_{13}}(D_3+D_3')\cong \cO_{C_{13}}(D_1')$. 
		Thus, (2) follows from
		\(\cK|_{C_{13}}\cong \cO_{C_{13}}(D_3+D_3')\otimes j_{31}^*\cM\).
		
		Similarly, since $\dHomk(\cE_3,\cE_2)=\C[-2]$ and $C_{23}\cong\PP^1$, then $\cO_{C_{23}}(D_3+D_3'-D_1')\cong \cO_{C_{23}}(-2)$ i.e., $\cO_{C_{23}}(D_3+D_3')\cong \cO_{C_{23}}(D_2'-2)$. 
		Thus, (3) is given by 
		\(\cK|_{C_{23}}\cong \cO_{C_{23}}(D_3+D_3')\otimes j_{32}^*\cM\).	
	\end{proof}
	
	\begin{lemma}\label{lemma.perpgen}
		Keep the conditions of Lemma~\ref{lemma.homexten} and assume $\dHomk(\cE_1,\cE_2)=\C[-1]$, $\cK|_{C_{13}}=\cO_{C_{13}}$ and $\cK|_{C_{23}}=\cO_{C_{23}}(-1)$. Then $T_{\cE_1}\cE_2\in \cE_3^\perp$.
	\end{lemma}
	\begin{proof}
		Using~\cite[Lemma~2.3(2)]{DZ}, we have the exact triangle
		\begin{equation}\label{seq.exten}
			\cE_2\rightarrow T_{\cE_1}\cE_2\rightarrow\cE_1
		\end{equation}
		Applying $\Hom^*(\cE_3,\sim)$ to \eqref{seq.exten}, we have $\Hom^*(\cE_3,\cE_2)=\C[-2]$ and $\Hom^*(\cE_3,\cE_1)=\C[-1]$, by Serre duality. We immediately see that $\Hom^i(\cE_3,T_{\cE_1}\cE_2)=0$ for $i\neq1,2$. By Euler characteristic, we have 
		\begin{equation*}
			\hom^1(\cE_3,T_{\cE_1}\cE_2)=\hom^2(\cE_3,T_{\cE_1}\cE_2).
		\end{equation*}
		By Lemma~\ref{lemma.homexten}, we only need to show,
		\[\Hom^1(\cE_3,T_{\cE_1}\cE_2)=H^0(C_{3},\cK)=0.\]
		
		We argue that $\Gamma(C_3,\cK)=0$ by considering an affine open containing the point $p$. Let $U\coloneqq C_3\setminus\{q_1,q_2\}$, with $ q_i\in C_{i3}$ and $q_i\neq p$, for $i=1,2$. Denote also $U_i=C_{i3}\setminus \{q_i\}$, for $i=1,2$. Then there are fibre squares
		\begin{equation}\label{eq.fibre}
			\begin{tikzcd}
				{C_{13}} && {C_3} && {C_{23}} \\
				{U_1} && {U} && {U_2}
				\arrow[hook, from=1-1, to=1-3]
				\arrow[hook', from=1-5, to=1-3]
				\arrow[hook, from=2-1, to=1-1]
				\arrow[hook, from=2-1, to=2-3]
				\arrow[hook, from=2-3, to=1-3]
				\arrow[hook', from=2-5, to=1-5]
				\arrow[hook', from=2-5, to=2-3]
			\end{tikzcd}
		\end{equation}
		The bottom row of (\ref{eq.fibre}) is affine, we can write it as
		\begin{equation*}
			\begin{tikzcd}
				\Spec \mathbf{k}[x] && \Spec \mathbf{k}[x,y]/xy && \Spec \mathbf{k}[y]
				\arrow[hook, from=1-1, to=1-3]
				\arrow[hook', from=1-5, to=1-3]
			\end{tikzcd}
		\end{equation*}
		\par
		Taking global sections of \eqref{eq.fibre} yields the commutative diagram
		\begin{equation*}
			\begin{tikzcd}
				{\mathbf{k}\simeq\Gamma_{C_{13}}\cK|_{C_{13}}} && {\Gamma_{C_{3}}\cK} && {\Gamma_{C_{23}}\cK|_{C_{23}}\simeq 0} \\
				{\mathbf{k}[x]\simeq\Gamma_{U_{1}}\cK|_{U_{1}}} && {\Gamma_{U}\cK|_{U}} && {\Gamma_{U_{2}}\cK|_{U_{2}}\simeq \mathbf{k}[y].}
				\arrow[from=1-1, to=2-1]
				\arrow[from=1-3, to=1-1]
				\arrow[from=1-3, to=1-5]
				\arrow[from=1-3, to=2-3]
				\arrow[from=1-5, to=2-5]
				\arrow[from=2-3, to=2-1]
				\arrow[from=2-3, to=2-5]
			\end{tikzcd}
		\end{equation*}
		The left-hand map is the obvious inclusion. Hence for every section $s\in\Gamma_{C_3}\cK$, we
		get the following relation
		\begin{equation*}
			\begin{tikzcd}
				&& {s} \\
				{\mathbf{k}[x]\supset\mathbf{k}\ni c} && {s|_U} && {0\in\mathbf{k}[y]}
				\arrow[maps to, from=1-3, to=2-3]
				\arrow[maps to, from=2-3, to=2-1]
				\arrow[maps to, from=2-3, to=2-5]
			\end{tikzcd}
		\end{equation*}
		The left-hand arrow gives that $s|_U\in\mathbf{k}+(y)$, and the right-hand arrow gives that $s|_U\in(x)$. But $(x)\cap(\mathbf{k}+(y))=0$, hence $s|_U=0$. Hence $s|_{U_i\setminus\{p\}}=0$ which implies $s|_{C_{i3}\setminus\{p\}}=0$ because $C_{i3}\setminus\{p\}\simeq\mathbb{A}^1$ is integral. But $\{C_{i3}\setminus\{p\}\}_{i=1,2} \cup \{U\}$ is an open cover of $C_3$, so $s=0$ by the sheaf condition.

		Thus $\Hom^1_X(\mathcal{E}_3,T_{\cE_1}\cE_2)=\Gamma(C_3,\cK)=0$.
		Thus $\Hom^*_X(\mathcal{E}_3,T_{\cE_1}\cE_2)=0$. Then by definition $T_{\cE_1}\cE_2\in\cE_3^\perp$.
		
	\end{proof}
	
\section{Basic notions of toric geometry and Hirzebruch surfaces}\label{sec.toric}
	\subsection{Toric varieties and line bundles}
	We recall some useful properties of toric geometry and Hirzebruch surfaces, following~\cite[Section~3]{DZ}.
	
	A \define{toric variety}~\cite[Definition~3.1.1]{CLS} $X$ over $\C$ is a $d$-dimensional irreducible variety containing a dense open torus $T_N\cong (\C^*)^d$ such that the multiplication action on $T_N$ extends to an algebraic action $T_N\times X\to X$. Let $M=\Hom(T_N,\C^*)$ be an abelian group of characters of $T_N$ and $N=M^\vee$.\par 
	
	A \define{fan}~$\Sigma$ in $N_{\R}\coloneqq N\otimes_{\Z}\R$ is a finite collection of strongly convex rational polyhedral cones $\sigma\subset N_{\R}$ satisfying the usual compatibility conditions~\cite[Definition~3.1.2]{CLS}. We denote by $\Sigma(r)$ the set of $r$-dimensional cones in $\Sigma$, and write $\tau\prec \sigma$ if $\tau$ is a face of $\sigma$.\par 
	
	For any cone $\sigma\in \Sigma$, the associated affine variety is 
	\[U_\sigma = \operatorname{Spec}\!\big(\C[\sigma^\vee \cap M]\big),
	\qquad 
	\sigma^\vee = \{\, m\in M_\R \mid \langle m,n\rangle \ge 0 
	\text{ for all } n\in\sigma \,\}.
	\]
	If $\tau=\sigma_1\cap\sigma_2$ is a common face, then 
	$U_\tau = U_{\sigma_1}\cap U_{\sigma_2}$.
	Gluing these affine pieces gives the toric variety $X_\Sigma$, 
	associated with~$\Sigma$~\cite[after Definition~3.1.2]{CLS}. There is an \define{orbit-cone correspondence} \cite[Theorem~3.2.6]{CLS} between cones and $T_N$-orbits in $X_{\Sigma}$ as follows.
	\[
	\{\sigma\in\Sigma\}
	\;\Longleftrightarrow\;
	\{\text{$T_N$-orbits in } X_\Sigma\}
	\qquad
	\sigma \longmapsto O(\sigma)\cong\Hom_\Z(\sigma^\perp\cap M,\C^*)
	\]
	Moreover, $\dim O(\sigma)=d-\dim\sigma$.  
	Write $V(\sigma)\coloneqq\overline{O(\sigma)}$.  
	Then
	\[V(\sigma)=\bigcup_{\sigma\preceq\tau} O(\tau).\]

	When $X_\Sigma$ is smooth, for any ray $\rho\in \Sigma(1)$, we write $D_{\rho}\in \mathbf{Pic}(X_\Sigma)$ for the torus-invariant divisor corresponding to $V(\rho)$. The canonical divisor is given by $K_{X_\Sigma}=-\sum_{\rho\in\Sigma(1)}D_{\rho}$.
	
	For a Cartier divisor $D=\sum_{\rho\in\Sigma(1)} a_\rho D_\rho$
	on $X_\Sigma$, we construct a new fan 
	$\Sigma\times D$ in $N_\R\times\R$ such that $X_{\Sigma\times D}$ is isomorphic to the total space of the line bundle $\mathcal{O}_{X_\Sigma}(D)$ with bundle projection induced by projection away from $\R$~\cite[Chapter~7.3]{CLS}. Namely, for each cone $\sigma\in\Sigma$, let
	\begin{equation}\label{eq.timesD}
		\tilde\sigma
		\coloneqq \big\langle (\underline{0},1),\,(u_\rho,-a_\rho)\big| \rho\preceq\sigma,\,\rho\in\Sigma(1) \big\rangle,
	\end{equation}
	where $u_\rho$ denotes the primitive generator of $\rho$. Then $\Sigma\times D$ is given by the $\tilde\sigma$ for all $\sigma\in\Sigma$, along with all their faces. 
	
	\begin{proposition}[\hspace*{-0.4em}{\cite[Proposition~3.3]{DZ}}]\label{prop.timesD}
		The subvariety $V((\underline{0},1))$
		is the zero section of $X_{\Sigma\times D}$. Writing $i$ for its embedding, $V(\tilde\sigma)=i\,V(\sigma)$ for any $\sigma\in\Sigma$.
	\end{proposition}

	\subsection{Hirzebruch surfaces and the blow-up of Hirzebruch surfaces}
	The Hirzebruch surface $\mathbb{F}_e$ is the projective bundle $\mathbb{P}(\cO_{\mathbb{P}^1}(-e)\oplus\cO_{\mathbb{P}^1})$ on ${\mathbb{P}^1}$. Denote by $C_0$ the projectivization of the section~$(0,1)$ in $\mathbb{F}_e$.
	According to~\cite[Example~4.1.8]{CLS}, the fan of $\mathbb{F}_e$ is given by Figure~\ref{fig.hizfan}, whose rays are denoted by $v_l$ for $l=1,\dots,4$. 
		\begin{figure}[htbp]
		\centering
		\begin{tikzpicture}[scale=0.5]
			\coordinate (V0) at (0,0);
			\coordinate (V1) at (2,0);
			\coordinate (V2) at (0,3);
			\coordinate (V3) at (0,-3);
			\coordinate (V4) at (-1,3);
			
			\draw (V1) -- (V0); 
			\draw (V2) -- (V0); 
			\draw (V3) -- (V0); 
			\draw (V4) -- (V0); 
			
			\node[circle,fill,inner sep=1pt,label={-45:$\mathbb{F}_e$}] at (V0) {};
			
			\node[right] at (V1) {$v_3(1,0)$};
			\node[above right] at (V2) {$v_2(0,1)$};
			\node[below right] at (V3) {$v_4(0,-1)$};
			\node[left]        at (V4) {$v_1(-1,e)$};
		\end{tikzpicture}  
		\caption{The fan of $\mathbb{F}_e$}
		\label{fig.hizfan}
	\end{figure}
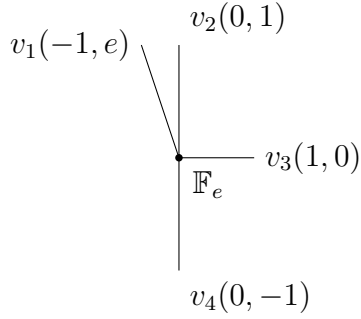
	
	Using \cite[Proposition~6.3.8]{CLS}, the intersection numbers are given as follows:
		\begin{itemize}
			\item $V(v_2)=C_0$ with $C_0^2=-e$ in $\mathbb{F}_e$.
			\item For $l=1,3$, $V(v_l)$ is a fiber of $\mathbb{F}_e$, satisfying $V(v_l)^2=0$ and $V(v_l)\cdot C_0=1$.
			\item $V(v_4)^2=e$, $V(v_4)\cdot C_0=0$, and $V(v_4)\cdot V(v_l)=1$ for $l=1,3$.
		\end{itemize}
		
		Similarly, the fan of $\mathbb{F}_e(1)$, the blow-up of $\mathbb{F}_e$ at one point, is obtained by adding the ray $v_5=v_3+v_4$ of the fan of $\F_e$ as shown in Figure~\ref{fig.hizfan(1)}. The fan of $\mathbb{F}_e(2)$, the blow-up of $\F_e$ at two points, is given by adding the ray $v_6=v_1+v_4$ of the fan of $\F_e(1)$ as shown in Figure~\ref{fig.hizfan(2)}.
	\begin{figure}[htbp]
		\centering
		\begin{minipage}[b]{0.48\textwidth}
			\centering
			\begin{tikzpicture}[scale=0.5]
				\coordinate (V0) at (0,0);
				\coordinate (V1) at (2,0);
				\coordinate (V2) at (0,3);
				\coordinate (V3) at (0,-3);
				\coordinate (V4) at (-1,3);
				\coordinate (V5) at (2,-2);
				
				\draw (V1) -- (V0); 
				\draw (V2) -- (V0); 
				\draw (V3) -- (V0); 
				\draw (V4) -- (V0); 
				\draw (V5) -- (V0);
				
				\node[circle,fill,inner sep=1pt,label={-135:$\mathbb{F}_e(1)$}] at (V0) {};
				
				\node[right] at (V1) {$v_3(1,0)$};
				\node[above right] at (V2) {$v_2(0,1)$};
				\node[below right] at (V3) {$v_4(0,-1)$};
				\node[left]        at (V4) {$v_1(-1,e)$};
				\node[below right] at (V5) {$v_5(1,-1)$};
			\end{tikzpicture} 
			\caption{The fan of $\mathbb{F}_e(1)$}
			\label{fig.hizfan(1)}
		\end{minipage}
		\hfill
		\begin{minipage}[b]{0.48\textwidth}
			\centering
			\begin{tikzpicture}[scale=0.5]
				\coordinate (V0) at (0,0);
				\coordinate (V1) at (2,0);
				\coordinate (V2) at (0,3);
				\coordinate (V3) at (0,-2);
				\coordinate (V4) at (-1,3);
				\coordinate (V5) at (2,-2);
				\coordinate (V6) at (-1,1);
				
				\draw (V1) -- (V0); 
				\draw (V2) -- (V0); 
				\draw (V3) -- (V0); 
				\draw (V4) -- (V0); 
				\draw (V5) -- (V0); 
				\draw (V6) -- (V0); 
				
				\node[circle,fill,inner sep=1pt,label={-135:$\mathbb{F}_e(2)$}] at (V0) {};
				
				\node[right] at (V1) {$v_3(1,0)$};
				\node[above right] at (V2) {$v_2(0,1)$};
				\node[below right] at (V3) {$v_4(0,-1)$};
				\node[left]        at (V4) {$v_1(-1,e)$};
				\node[below right] at (V5) {$v_5(1,-1)$};
				\node[left]        at (V6) {$v_6(-1,e-1)$};
		\end{tikzpicture}  
			\caption{The fan of $\mathbb{F}_e(2)$}
			\label{fig.hizfan(2)}
		\end{minipage}
	\end{figure}
	
	Also using \cite[Proposition~6.3.8]{CLS}, the self-intersection numbers in $\F_e(1)$ are given as follows:
	\begin{itemize}
		\item $V(v_1)^2=0$ in $\mathbb{F}_e(1)$.
		\item $V(v_2)^2=-e$ in $\mathbb{F}_e(1)$.
		\item For $l=3,5$, $V(v_l)^2=-1$ in $\mathbb{F}_e(1)$.
		\item $V(v_4)^2=e-1$ in $\mathbb{F}_e(1)$.
	\end{itemize}
	The self-intersection numbers in $\F_e(2)$ are given as follows:
	\begin{itemize}
		\item $V(v_2)^2=-e$ in $\mathbb{F}_e(2)$.
		\item For $l=1,3,5,6$, $V(v_l)^2=-1$ in $\mathbb{F}_e(2)$.
		\item $V(v_4)^2=e-2$ in $\mathbb{F}_e(2)$.
	\end{itemize}

\section{Exceptional surfaces in $G\text{-Hilb}(\C^3)$ and their intersections}\label{sec.exc}
	
	Let $G\subset \SL(3,\C)$ be a diagonal cyclic subgroup acting on $\C^3$ and write elements of $G$ as $\diag(\xi^{a},\xi^{b},\xi^{c})$ where $\xi$ is a primitive $r^{\text{th}}$ root of unity. The lattice for the fan $\Sigma_{\C^3/G}$ of the quotient $\C^3/G$ is given by the overlattice $L\supset\Z^3$ generated by
	\[
	L = \Z^3 + \Big\langle \tfrac{1}{r}(a,b,c)\;\Big|\; \diag(\xi^a,\xi^b,\xi^c)\in G,\ a,b,c\ge 0 \Big\rangle.
	\]

	Write $G\text{-}\Hilb(\C^3)$ for the $G$-orbit Hilbert scheme over $\C^3$. It parametrizes $G$-invariant smoothable zero-dimensional subschemes of $\C^3$ of length $ |G|$~\cite{Nakamura}. For convenience, we set $X(1,s,r-s-1)=G\text{-}\Hilb(\C^3)$ for $G=\mu_r$ with weights $(1,s,r-s-1)$.\par
	
	According to~\cite{IR,Reid,CR,Craw}, the fan $\Sigma_{X(1,s,r-s-1)}$ of $X(1,s,r-s-1)$ is a refinement of $\Sigma_{\C^3/G}$ by inserting additional rays as follows.
	\begin{equation*}
		\Sigma_{X(1,s,r-s-1)}(1)=\Sigma_{\C^3/G}(1)\:\bigcup\:\big\{\tfrac{1}{r}(a,b,c)\big|
		\diag(\xi^{a},\xi^{b},\xi^{c})
		\in G,\ a+b+c=r\big\}
	\end{equation*}
	By~\cite{Nakamura}, the projective toric resolution 
	\begin{equation*}
		f\colon  X(1,s,r-s-1)\longrightarrow \C^3/G
	\end{equation*}
	given by this refinement is a projective crepant resolution, hence $X(1,s,r-s-1)$ is Calabi-Yau.\par
	
	By~\cite[Proposition~11.1.10]{CLS}, the irreducible compact exceptional surfaces in $X(1,s,r-s-1)$ are precisely the torus-invariant divisors $V(\rho)$ corresponding to the new rays
	\[
	\rho \in \Sigma_{X(1,s,r-s-1)}(1) \setminus \Sigma_{\C^3/G}(1).
	\]
	
	Any three 1-dimensional cones $\rho_1,\rho_2,\rho_3\in \Sigma_{X(1,s,r-s-1)}(1)$ such that $\langle \rho_1,\rho_2,\rho_3\rangle$ forms a 3-dimensional cone generate the lattice $L$. Moreover, for each $k=1,2,3$ there exists a lattice isomorphism $\phi_k\colon L\to \Z^3$ sending $\rho_k$ to $e_3=(0,0,1)$. This yields a local coordinate system in which the three-dimensional cone can be identified with the standard positive orthant in $\Z^3$.
	
	Following~\cite[Section~4]{DZ}, we consider the subfan $\tilde\Sigma_S$ consisting of all cones in $\Sigma_{X(1,s,r-s-1)}$ that have $\rho_S$ as a face. 
	There is a lattice isomorphism $\varphi_S\colon  L\to \Z^3$ such that $\varphi_S((a_S,b_S,c_S))=(0,0,1)$, and $\tilde\Sigma_S$ is compatible with $\Sigma_S\times K_S$ defined in \eqref{eq.timesD} under $\varphi$, where $\Sigma_S$ is the fan of $S$ in $\Z^2$ and $K_S$ is the torus-invariant canonical divisor of $S$.

	 By~\cite[Corollary~1.6]{CR}, there are four types of subfans.
	 \begin{proposition}\label{prop.exceptional types}
	 	Irreducible compact exceptional surfaces in $X(1,s,r-s-1)$ can be determined as follows:
	 	\begin{itemize}
	 		\item[(1)] if there are three 3-dimensional cones in $\Sigma_S$, $S$ is isomorphic to $\PP^2$;
	 		\item[(2)] if there are four 3-dimensional cones in $\Sigma_S$, $S$ is isomorphic to $\F_e$;
	 		\item[(3)] if there are five 3-dimensional cones in $\Sigma_S$, $S$ is isomorphic to $\F_e(1)$;
	 		\item[(4)] if there are six 3-dimensional cones in $\Sigma_S$, $S$ is isomorphic to $\F_e(2)$.
	 	\end{itemize}
	 \end{proposition}
	 
	 For two irreducible compact exceptional surfaces $S_1$ and $S_2$, similar to \cite[Proposition~4.4]{DZ}, by using Proposition~\ref{prop.timesD}, the non-empty intersection $C_{S_1,S_2}\coloneqq V(\rho_{S_1})\cap V(\rho_{S_2})$ is isomorphic to $V(\langle\rho_{S_1},\rho_{S_2}\rangle) \cong V(\tau_{S_1}) \cong V(\tau_{S_2}) \cong \PP^1$. Here, $\tau_{S_1}$ is the 1-dimensional cone in $\Sigma_{S_1}$ such that $\varphi_{S_1}(\langle \rho_{S_1},\rho_{S_2}\rangle)=\widetilde\tau_{S_1}$, and $\tau_{S_2}$ is the 1-dimensional cone in $\Sigma_{S_2}$ such that $\varphi_{S_2}(\langle \rho_{S_1},\rho_{S_2}\rangle)=\widetilde\tau_{S_2}$.
	 In addition, $C_{S_1,S_2}^2$ in $S_k$ is $V(\tau_{S_k})^2$ for $k=1,2$. 
	
	Thus, if we know an irreducible compact exceptional surface in $X(1,s,r-s-1)$, then the remaining exceptional surfaces in $X(1,s,r-s-1)$ are determined by Lemma~\ref{lemma.sum} and Proposition~\ref{prop.exceptional types}.
	
	\subsection{Example}
	For example, the junior simplex for $X(1,3,9)$ is shown in Figure~\ref{fig.js-13cycle} and $\rho_k$ is a 1-dimensional cone in $\Sigma_{X(1,3,9)}$.
	\begin{figure}[htbp]
		\begin{center}
			
			\tdplotsetmaincoords{-145}{90}
			
			\begin{tikzpicture}[tdplot_main_coords,scale=1/2]
				
				\tdplotsetrotatedcoords{-90}{45}{90}
				
				\begin{scope}[tdplot_rotated_coords]
					\coordinate (O) at (0,0,0);
					\coordinate (P1) at (13,0,0); 
					\coordinate (P2) at (0,13,0); 
					\coordinate (P3) at (0,0,13);
					\coordinate (Q1) at (9,1,3);
					\coordinate (Q2) at (3,9,1);
					\coordinate (Q3) at (1,3,9);
					\coordinate (R1) at (5,2,6);
					\coordinate (R2) at (6,5,2);
					\coordinate (R3) at (2,6,5);
				\end{scope}
				
				\draw (P1) -- (P2) -- (P3) -- cycle; 
				\draw (R1) -- (R2) -- (R3) -- cycle; 
				\draw (P1) -- (Q3);
				\draw (P2) -- (Q1);
				\draw (P3) -- (Q2);
				\draw (P1) -- (Q2);
				\draw (P2) -- (Q3);
				\draw (P3) -- (Q1);
				\draw (P1) -- (R2);
				\draw (P2) -- (R3);
				\draw (P3) -- (R1);
				
				\node[circle,fill,inner sep=1pt,label={90:$\rho_7(0,0,1)$}] at (P1) {};
				\node[circle,fill,inner sep=1pt,label={0:$\rho_8(1,0,0)$}] at (P2) {};
				\node[circle,fill,inner sep=1pt,label={180:$\rho_9(0,1,0)$}] at (P3) {};
				\node[circle,fill,inner sep=1pt,label={-90:$\rho_3(\frac{3}{13},\frac{9}{13},\frac{1}{13})$}] at (Q3) {};
				\node[circle,fill,inner sep=1pt,label={0:$\rho_5(\frac{9}{13},\frac{1}{13},\frac{3}{13})$}] at (Q2) {};
				\node[circle,fill,inner sep=1pt,label={0:$\rho_1(\frac{1}{13},\frac{3}{13},\frac{9}{13})$}] at (Q1) {};
				\node[circle,fill,inner sep=1pt,label={-90:$\rho_4(\frac{6}{13},\frac{5}{13},\frac{2}{13})$}] at (R3) {};
				\node[circle,fill,inner sep=1pt,label={-90:$\rho_6(\frac{5}{13},\frac{2}{13},\frac{6}{13})$}] at (R2) {};
				\node[circle,fill,inner sep=1pt,label={-90:$\rho_2(\frac{2}{13},\frac{6}{13},\frac{5}{13})$}] at (R1) {};
				
			\end{tikzpicture}
		\end{center}
		\caption{The junior simplex of $X(1,3,9)$}
		\label{fig.js-13cycle}
	\end{figure}
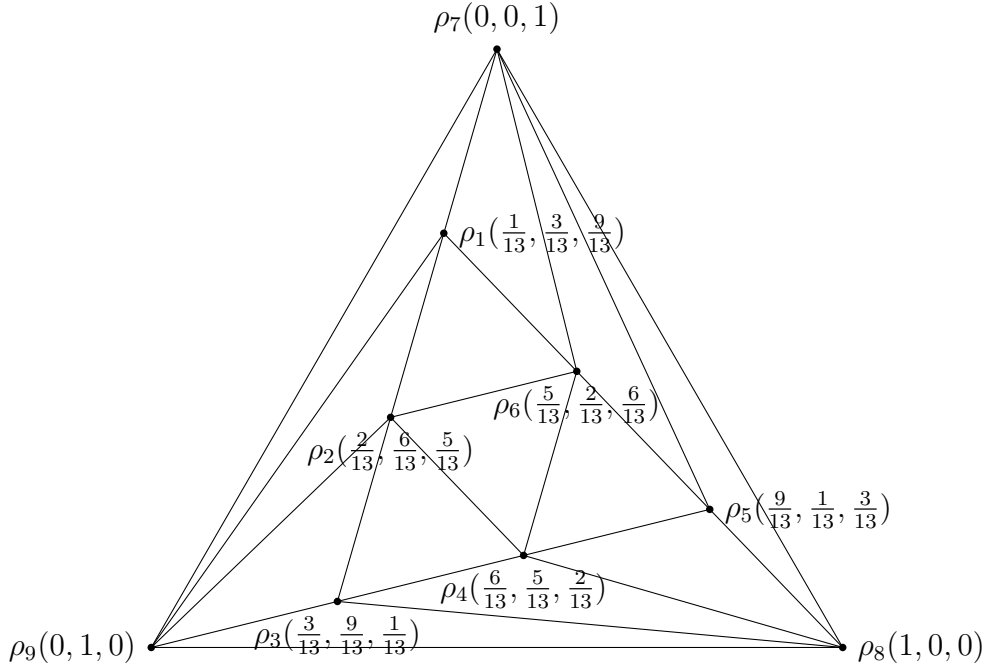
	
	Denote exceptional surfaces in $X(1,3,9)$ by $S_k \coloneqq V(\rho_k)$ for $1\leq k\leq 6$, with closed embeddings $i_k\colon S_k\to X(1,3,9)$. Let $C_{k,l} \coloneqq S_k\cap S_l \cong V(\langle\rho_k,\rho_l\rangle)$ whenever $S_k\cap S_l\neq\varnothing$.
	
	\begin{proposition}\label{prop.excsurfin139}
		The irreducible compact exceptional surfaces in $X(1,3,9)$ are as follows:
		\begin{itemize}
			\item[(1)] $S_{2k}\cong \F_2(1)$, for $k=1,2,3$;
			\item[(2)] $S_{2k-1}\cong \F_3$, for $k=1,2,3$.  
		\end{itemize}
	\end{proposition}
	\begin{proof}
		By rotational symmetry, we only need to show
		$S_4\cong \F_2(1)$ and $S_3\cong\F_3$.
		
		Denote the subfan generated by $\{\rho_2,\rho_3,\rho_4,\rho_5,\rho_6,\rho_8\}$ as $\Sigma_4$. The junior simplex $\triangle_4$ is then given by $\Sigma_4 \cap \{x+y+z=1\}$, as shown in Figure~\ref{fig.js-subfan}. The fan of $\F_2(1)$ is shown in Figure~\ref{fig.H2(1)}.
		\begin{figure}[htbp]
			\centering
			\begin{minipage}[b]{0.45\textwidth}
				\begin{center}
					
					\tdplotsetmaincoords{-145}{90}
					
					\begin{tikzpicture}[tdplot_main_coords,scale=0.5]
						
						\tdplotsetrotatedcoords{-90}{45}{90}
						
						\begin{scope}[tdplot_rotated_coords]
							\coordinate (O) at (0,0,0);
							
							\coordinate (P2) at (0,13,0); 
							\coordinate (Q2) at (3,9,1);
							\coordinate (Q3) at (1,3,9);
							\coordinate (R1) at (5,2,6);
							\coordinate (R2) at (6,5,2);
							\coordinate (R3) at (2,6,5);
						\end{scope}

						\draw (R1) -- (R2) -- (R3) -- cycle; 
						\draw (Q3) -- (Q2);
						\draw (P2) -- (Q3);
						\draw (P2) -- (R3);
						\draw (P2) -- (R2);
						\draw (Q3) -- (R1);


						\node[circle,fill,inner sep=1pt,label={270:$\rho_8(1,0,0)$}] at (P2) {};
						\node[circle,fill,inner sep=1pt,label={-90:$\rho_3(\frac{3}{13},\frac{9}{13},\frac{1}{13})$}] at (Q3) {};
						\node[circle,fill,inner sep=1pt,label={0:$\rho_5(\frac{9}{13},\frac{1}{13},\frac{3}{13})$}] at (Q2) {};
						\node[circle,fill,inner sep=1pt,label={-90:$\rho_4(\frac{6}{13},\frac{5}{13},\frac{2}{13})$}] at (R3) {};
						\node[circle,fill,inner sep=1pt,label={0:$\rho_6(\frac{5}{13},\frac{2}{13},\frac{6}{13})$}] at (R2) {};
						\node[circle,fill,inner sep=1pt,label={90:$\rho_2(\frac{2}{13},\frac{6}{13},\frac{5}{13})$}] at (R1) {};
						
					\end{tikzpicture}
				\end{center}
				\caption{The junior simplex of subfan}
				\label{fig.js-subfan}
			\end{minipage}
			\hfill
			\begin{minipage}[b]{0.48\textwidth}
				\centering
				\begin{tikzpicture}[scale=2]
					\coordinate (V0) at (0,0);
					\coordinate (V1) at (0,1);
					\coordinate (V2) at (-1,0);
					\coordinate (V3) at (1,0);
					\coordinate (V4) at (2,-1);
					\coordinate (V5) at (-1,1);
					
					\draw (V1) -- (V0); 
					\draw (V2) -- (V0); 
					\draw (V3) -- (V0); 
					\draw (V4) -- (V0); 
					\draw (V5) -- (V0); 
					
					\node[above] at (V1) {$v_6(1,0)$};
					\node[below] at (V2) {$v_3(0,1)$};
					\node[right] at (V3) {$v_5(0,-1)$};
					\node[below left]  at (V4) {$v_8(-1,-2)$};
					\node[left]        at (V5) {$v_2(1,1)$};
				\end{tikzpicture}  
				\caption{The fan of $\F_2(1)$}
				\label{fig.H2(1)}
			\end{minipage}
		\end{figure}
		
		Denote $\rho_{4}'\coloneqq(0,0,1)$ and $\rho_{k}'\coloneqq (v_k,1)$ in $\Sigma_{\F_2(1)}\times K_{\F_2(1)}$ for $k=2,3,5,6,8$. Then identifying $\rho_k$ with $\rho_k'$ for $k=2,3,4$, there exists a lattice isomorphism $\phi\colon L_{(1,3,9)}\to \Z^3$ such that $\Sigma_4$ and $\Sigma_{\F_2(1)}\times K_{\F_2(1)}$ are compatible.
		Since
		\[\rho_8-\rho_4=(\frac{7}{13},\frac{-5}{13},\frac{-2}{13}),\quad
		\rho_5-\rho_4=(\frac{3}{13},\frac{-4}{13},\frac{1}{13}),\quad
		\rho_6-\rho_4=(\frac{-1}{13},\frac{-3}{13},\frac{4}{13}),\]
		we have
		\begin{equation*}
			\rho_8-\rho_4=2(\rho_5-\rho_4)-(\rho_6-\rho_4).
		\end{equation*}
		
		Thus, using \cite{DZ}, $S_4=V(\rho_4)\cong V(\rho_4')\cong \F_2(1)$, and $C_{4,k}\cong V(\langle\rho_4,\rho_k\rangle)\cong V(\langle\rho_4',\rho_k'\rangle)\cong V(v_k)$ for $k=2,3,5,6,8$.
		
		By this isomorphism, we have $C_{34}^2=1$ in $S_4$. Then we have $C_{34}^2=-3$ in $S_3$ using Lemma~\ref{lemma.sum}. Thus, $S_3\cong \F_3$ using~\cite[Corollary~1.6]{CR}, since there are $4$ 2-dimensional cones containing $\rho_3$.
	\end{proof}
	
	Using Lemma~\ref{lemma.sum} and  Proposition~\ref{prop.exceptional types}, we get the following.
	
	\begin{proposition}\label{prop.inter.139}
		The following are self-intersection numbers for all non-empty $C_{t,l}$.
		\begin{itemize}
			\item $C_{2k-1,2k}^2=-3$ in $S_{2k-1}$ and $C_{2k-1,2k}^2=1$ in $S_{2k}$, for $k=1,2,3$;
			\item $C_{2k,2l}^2=-1$ in $S_{2k}$ and in $S_{2l}$; for $1\leq k\neq l\leq 3$;
			\item $C_{2k,2k+1}^2=-2$ in $S_{2k}$ and $C_{2k,2k+1}^2=0$ in $S_{2k+1}$ for $k=1,2,3$ with indices taken modulo $6$.
		\end{itemize}
	\end{proposition}
	
	Then we can draw the relation among all irreducible compact exceptional surfaces in $X(1,3,9)$, as shown in Figure~\ref{fig.rel139}. 
	\begin{figure}[htbp]
		\centering
		
		\begin{tikzpicture}[scale=1.4,font=\scriptsize]
			\hexsetup
			\hexagon{-\hs}{0}{S_1\cong\F_3}
			\hexagon{-2*\hs}{-1.5}{S_2\cong\F_2(1)}
			\hexagon{0}{-1.5}{S_6\cong\F_2(1)}
			\hexagon{-3*\hs}{-3}{S_3\cong\F_3}
			\hexagon{-\hs}{-3}{S_4\cong\F_2(1)}
			\hexagon{\hs}{-3}{S_5\cong\F_3}
			\node[fill=white, inner sep=1pt, text=red] at (-8/3*\hs,-2.4) {0};  
			\node[fill=white, inner sep=1pt, text=red] at (-7/3*\hs,-2.1) {-2};
			\node[fill=white, inner sep=1pt, text=red] at (-0.15-2*\hs,-3) {-3};  
			\node[fill=white, inner sep=1pt, text=red] at (0.15-2*\hs,-3) {1};
			\node[fill=white, inner sep=1pt, text=red] at (-5/3*\hs,-0.9) {1};  
			\node[fill=white, inner sep=1pt, text=red] at (-4/3*\hs,-0.6) {-3};
			\node[fill=white, inner sep=1pt, text=red] at (-0.15-\hs,-1.5) {-1};  
			\node[fill=white, inner sep=1pt, text=red] at (0.15-\hs,-1.5) {-1};
			\node[fill=white, inner sep=1pt, text=red] at (-4/3*\hs,-2.4) {-1};  
			\node[fill=white, inner sep=1pt, text=red] at (-5/3*\hs,-2.1) {-1};
			\node[fill=white, inner sep=1pt, text=red] at (-1/3*\hs,-0.9) {-2};  
			\node[fill=white, inner sep=1pt, text=red] at (-2/3*\hs,-0.6) {0};
			\node[fill=white, inner sep=1pt, text=red] at (-2/3*\hs,-2.4) {-1};  
			\node[fill=white, inner sep=1pt, text=red] at (-1/3*\hs,-2.1) {-1};
			\node[fill=white, inner sep=1pt, text=red] at (2/3*\hs,-2.4) {-3};  
			\node[fill=white, inner sep=1pt, text=red] at (1/3*\hs,-2.1) {1};
			\node[fill=white, inner sep=1pt, text=red] at (-0.15,-3) {-2};  
			\node[fill=white, inner sep=1pt, text=red] at (0.15,-3) {0};
		\end{tikzpicture}
		\caption{Exceptional surfaces in $X(1,3,9)$ and their intersections}
		\label{fig.rel139}
	\end{figure}
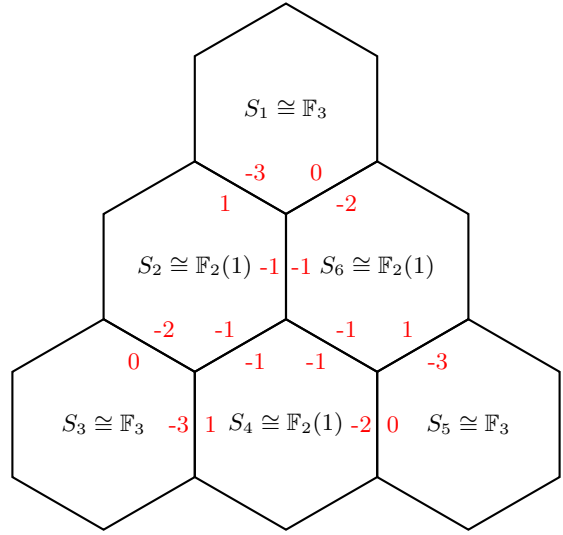
	
	In Figure~\ref{fig.rel139}, each hexagon is an irreducible compact exceptional surface $S_k$, an edge represents a $\mathbb{P}^1$ intersection curve $C_{k,l}$, and the red number adjacent to an edge inside hexagon $S_k$ is $C_{k,l}^2$ in $S_k$. A vertex where three hexagons meet indicates a triple intersection point. For example, the red number ``$-3$" in the hexagon $S_1$ near the hexagon $S_2$ means $C_{12}^2=-3$ in $S_1$. The shared point among three hexagons $S_1$, $S_2$ and $S_6$ means $S_1\cap S_2\cap S_6=\{\text{a point}\}$.

	\section{The faithful algebraic braid twist group action associated with exceptional surfaces in $X(1,3,9)$}\label{sec.139}
			
		Denote $\cE_{2l-1}\coloneqq i_{2l-1*}\cO_{S_{2l-1}}$ and $\cE_{2l}\coloneqq i_{2l*}\cO_{S_{2l}}(-C_{2l-1,2l})$ for $l=1,2,3$. Each $\cE_k$ is spherical in $D(X(1,3,9))$ by Proposition~\ref{prop.sph} and Lemma~\ref{lemma.sph}. Denote $T_k\coloneqq T_{\cE_k}$, the spherical twist associated with $\cE_k$. This yields an algebraic braid twist group action on $D(X(1,3,9))$.
		
	\begin{theorem}\label{thm.D6}
		The algebraic braid twist group $\AT(Q,W)$ associated with the quiver $(Q,W)$ in Figure~\ref{fig.quiver139} acts on $D(X(1,3,9))$ faithfully via $\beta_k\mapsto T_k$.
	\end{theorem}

	\begin{proposition}\label{prop.139A2}
		We have
		\begin{itemize}
			\item $\dHomk(\cE_{2k-1},\cE_{2k})=\C[-1]$ for $k=1,2,3$;
			\item $\dHomk(\cE_{2k+2},\cE_{2k})=\C[-1]$ for $k=1,2,3$ with indices taken modulo $6$;
			\item $\dHomk(\cE_{2k},\cE_{2k+1})=\C[-1]$ for $k=1,2,3$ with indices taken modulo $6$.
			\end{itemize}
		\end{proposition}
	\begin{proof}
		By rotational symmetry, we need to show 
		\begin{itemize}
			\item[(1)] $\dHomk(\cE_1,\cE_{2})=\C[-1]$;
			\item[(2)] $\dHomk(\cE_{2},\cE_6)=\C[-1]$;
			\item[(3)] $\dHomk(\cE_6,\cE_1)=\C[-1]$;
		\end{itemize}
		Since $\cE_2\cong i_{2*}\cO_{S_2}\otimes\cO_{X(1,3,9)}(-S_1)$, the statement (1) follows from Corollary~\ref{coro.Hom}.
			
		Using Corollary~\ref{coro.Homspace}, we have 
		\[\dHomk(\cE_{2},\cE_6)\cong \Gamma_{C_{26}}^\bullet\cO_{C_{26}}(S_2+S_1-S_5)[-1]
		\cong \Gamma_{C_{26}}^\bullet\cO_{C_{26}}[-1]=\C[-1],\]
		where the last isomorphism uses $S_5\cap C_{26}=\varnothing$, $S_1\cap C_{26}=\{\text{a point}\}$, Lemma~\ref{lemma.number}, and Proposition~\ref{prop.inter.139}. 
			
		Again by Corollary~\ref{coro.Homspace}, we have
		\begin{align*}
			\dHomk(\cE_6,\cE_1)\cong \Gamma_{C_{16}}^\bullet\cO_{C_{16}}(S_6+S_5)[-1]\cong \Gamma_{C_{16}}^\bullet\cO_{C_{16}}[-1]=\C[-1].
		\end{align*}
		where the last isomorphism uses $ S_5\cap C_{16}=\varnothing$, Lemma~\ref{lemma.number}, and Proposition~\ref{prop.inter.139}.
		\end{proof}
		
		\begin{proposition}\label{prop.139xzyperp}
			$T_{k}\cE_{k+1}\in \cE_{k-1}^\perp$ for $k=1,3,5$ with indices taken modulo $6$.
		\end{proposition}
		\begin{proof}
			Again by symmetry, we only need to show $T_{3}\cE_{4}\in \cE_{2}^\perp$.
			
			Denote $S_{34}\coloneqq S_3\cup S_4$ with $i_{34}\colon S_{34}\to X(1,3,9)$. Write  $C_{2}\coloneqq C_{23}\cup C_{24}$, with $j_{2}\colon C_{2}\to S_{34}$.

			Using Proposition~\ref{prop.inter.139} and Lemma~\ref{lemma.extcon}, we obtain
			\[T_{3}\cE_{4}\cong {i_{34*}}\cO_{S_{34}}.\]
			
			By Proposition~\ref{prop.139A2} and Lemma~\ref{lemma.homexten}(1),
			\[\mathcal{H}om_{X(1,3,9)}(\mathcal{F}_{2}, T_{3}\cE_{4})
			\cong  i_{{34}*}j_{2*}\cK[-1].\]
			where $\cK=\mathcal{O}_{C_{2}}(S_2+S_1)\otimes j_{2}^*\cM$ and $\cM=\cO_{S_{34}}$. 
			
			Using Lemma~\ref{lemma.homexten}(2), we have \(\cK|_{C_{23}}\cong \mathcal{O}_{C_{23}}\).
			
			Using Lemma~\ref{lemma.homexten}(3), we have 
			$\cK|_{C_{24}}
			\cong \cO_{C_{24}}(S_3-2)
			\cong \cO_{C_{24}}(-1)$,
			where the last isomorphism uses $S_3\cap C_{24}=\{\text{a point}\}$. 
			
			Thus, the claim follows from Lemma~\ref{lemma.perpgen}.
		\end{proof} 
		
		Denote $S_{24}\coloneqq S_2\cup S_4$ and the closed embedding by $i_{24}\colon S_{24}\to X(1,3,9)$. Denote $C_6=C_{46}\cup C_{26}$ and the closed embedding by $j_6\colon  C_6\to  S_{24}$. Denote $F_4\coloneqq \{x=0\}\cap S_4$.
		
		\begin{lemma}\label{lemma.divisorsim}
			In $S_4$, $C_{24}$ is linearly equivalent to $2F_{4}+C_{45}-C_{46}$ as divisors.
		\end{lemma}
		\begin{proof}
			By \cite[Proposition~4.1.2]{CLS}, in $\F_2(1)$ as shown in Figure~\ref{fig.H2(1)}, we have
			\begin{equation*}
				0\sim \mathrm{div}(\chi^{e_1})=V(v_3)+V(v_2)-2V(v_8)-V(v_5).
			\end{equation*}
			
			Using the definition of $\Sigma_{\F_2(1)}\times K_{\F_2(1)}$, we have
			\[0\sim \mathrm{div}(\chi^{e_1})=V(\tilde v_3)+V(\tilde v_2)-2V(\tilde v_8)-V(\tilde v_5).\]
			
			Using the toric isomorphism induced by the lattice isomorphism in Proposition~\ref{prop.excsurfin139}, we have
			\begin{equation*}
				0\sim V(\langle\rho_3,\rho_4\rangle)+V(\langle\rho_2,\rho_4\rangle)-2V(\langle\rho_4,\rho_8\rangle)-V(\langle\rho_4,\rho_5\rangle).
			\end{equation*}
			The claim then follows from the identifications $V(\langle\rho_k,\rho_l\rangle)=C_{k,l}$ and $V(\rho_8)=\{x=0\}$. 
		\end{proof}
		
		\begin{lemma}\label{lemma.F2F1}
			$T_{4}\cE_{2}$ is isomorphic to $i_{24*}\cO_{ S_{24}}(C_{24}-C_{12}-2F_4-C_{45})$.
		\end{lemma}
		\begin{proof}
			In fact, there is an exact sequence
			\begin{equation}\label{ext.s24}
				0\longrightarrow\mathcal{I}_{ S_2|C_{24}}\longrightarrow\cO_{ S_{24}}\longrightarrow\cO_{ S_4}\longrightarrow0,
			\end{equation}
			on $ S_{24}$, where the extension is non-split since $ S_2\cap  S_4$ is non-empty.
			
			Tensoring \eqref{ext.s24} with $\cO_{ S_{24}}(C_{24}-C_{12})$, we get the short exact sequence
			\begin{equation}\label{ext.s24c24}
				0\longrightarrow\cO_{ S_2}(-C_{12})\longrightarrow\cO_{ S_{24}}(C_{24}-C_{12})\longrightarrow\cO_{ S_4}(C_{24})\longrightarrow0,
			\end{equation}
			where the right-hand term follows from $C_{12}\cap S_4=\emptyset$.
			
			By Lemma~\ref{lemma.divisorsim}, $\cO_{ S_4}(C_{24})\cong\cO_{ S_4}(-C_{34}+2F_4+C_{45})$ on $ S_4$. Then the right part of \eqref{ext.s24c24}  can be replaced.
			\begin{equation}\label{ext.s24c34}
				0\longrightarrow\cO_{ S_2}(-C_{12})\longrightarrow\cO_{ S_{24}}(C_{24}-C_{12})\longrightarrow\cO_{ S_4}(-C_{34}+2F_4+C_{45})\longrightarrow0.
			\end{equation}
			Applying $\cO_{ S_{24}}(-2F_4-C_{45})$ to \eqref{ext.s24c34}, we get the short exact sequence
			\begin{equation}\label{ext.F1F2(1)}
				0\longrightarrow\cO_{ S_2}(-C_{12})\longrightarrow\cO_{ S_{24}}(C_{24}-C_{12}-2F_4-C_{45})\longrightarrow\cO_{ S_4}(-C_{34})\longrightarrow0,
			\end{equation}
			where the left-hand term follows from $C_{45}\cap S_2=2F_4\cap S_2=\emptyset$. After applying $i_{{24}*}$ on \eqref{ext.F1F2(1)}, we  get the exact triangle in $D(X(1,3,9))$
			\begin{equation*}
				\cE_{2}\longrightarrow i_{24*}\cO_{ S_{24}}(C_{24}-C_{12}-2F_4-C_{45})\longrightarrow \cE_{4}.
			\end{equation*}
			Comparing with~\cite[Theorem~2.3(2)]{DZ}, the claim is proved. 
		\end{proof}

		\begin{proposition}\label{prop.139xyzperp}
			$T_4\cE_2$ lies in the right orthogonal of $\cE_6$, i.e.,~
			$T_{4}\cE_{2}\in\cE_6^\perp$.
		\end{proposition}
		\begin{proof}
			By Proposition~\ref{prop.139A2}, Lemma~\ref{lemma.homexten}(1) and~\ref{lemma.F2F1},
			\[\mathcal{H}om_{X(1,3,9)}(\cE_{6}, T_{4}\cE_{2})
			\cong  i_{{24}*}j_{6*}\cK[-1].\]
			where $\cK=\mathcal{O}_{C_{6}}(S_6+S_5)\otimes j_{6}^*\cM$ and $\cM=\cO_{ S_{24}}(C_{24}-C_{12}-2F_4-C_{45})$.
			
			Using Lemma~\ref{lemma.homexten}(2), we have \[\cK|_{C_{46}}\cong \mathcal{O}_{C_{46}}(C_{24}\cdot C_{46}-C_{24}\cdot C_{45}-2F_4\cdot C_{46})\cong\mathcal{O}_{C_{46}},\]
			 because $C_{24}\cdot C_{46}=1=C_{46}\cdot C_{45}$ and $F_4\cdot C_{46}=0$ in $S_4$ and $C_{46}\cap S_3=\varnothing$, and $C_{12}\cap C_{46}=\varnothing$.

			Also using Lemma~\ref{lemma.homexten}(3), we have \[\cK|_{C_{26}}\cong \cO_{C_{26}}(C_{26}\cdot C_{12}-C_{26}\cdot C_{24})(S_1-2)
			\cong \cO_{C_{26}}(-1),\] since $C_{26}\cdot C_{12}=C_{26}\cdot C_{24}=1$ in $S_2$ and $C_{26}\cap S_1=\{\text{a point}\}$. 
			
			Thus, the claim is proved by Lemma~\ref{lemma.perpgen}.
		\end{proof}
		
		\begin{proof}[Proof of Theorem~\ref{thm.D6}]
			By Proposition~\ref{prop.139A2},~\ref{prop.139xyzperp} and~\ref{prop.139xzyperp}, $\{\cE_k\}_{1\leq k\leq 6}$ is a $(Q,W)$-configuration in $D(X(1,3,9))$, since
			$\dHomk(\cE_k,\cE_l)=0$ for $S_k\cap S_l=\varnothing$.
			Thus, the action via $\beta_k\mapsto T_k$ is given by Proposition~\ref{prop.cycle}.
			
			To prove faithfulness, by Proposition~\ref{prop.D6}, \[\{\cF_1\coloneqq\cE_5,\,\cF_2\coloneqq T_4T_5T_2(\cE_6),\,\cF_3\coloneqq\cE_4,\,\cF_4\coloneqq\cE_3,\,\cF_5\coloneqq T_2(\cE_3),\, \cF_6\coloneqq \cE_1\}\] is a $D_6$-configuration in $D(X(1,3,9))$ as shown in Figure~\ref{fig.D6con}. Using \cite[Theorem~1]{NV}, the $D_6$-configuration induces a faithful $\Br(D_6)$-action via $T'_k\coloneqq T_{\cF_k}$. Using Lemma~\ref{lemma.equivalence}(2), we have 
			\[T'_1\cong T_5,\quad
			T'_2\cong (T_4T_5T_2)T_6(T_4T_5T_2)^{-1},\quad
			T'_3\cong T_4,\quad
			T'_4\cong T_3,\quad
			T'_5\cong T_2T_3T_2^{-1},\quad
			T'_6\cong T_1.\]
			Conversely, we have
			\[T_1\cong T'_6,\quad
			T_2\cong T'_4T'_5{T'_4}^{-1},\quad
			T_3\cong T'_4,\quad
			T_4\cong T'_3,\quad
			T_5\cong T'_1,\quad
			T_6\cong (T'_3T'_4T'_5)^{-1}T'_2(T'_3T'_4T'_5).\]
			Thus, the action of $\AT(Q,W)$ is also faithful.
		\end{proof}
		
	\section{The faithful algebraic braid twist group action associated with exceptional surfaces in $X(1,3,13)$}\label{sec.1313}
	
		\subsection{Exceptional surfaces in $X(1,3,13)$}
		
		The junior simplex for $X(1,3,13)$ is shown in Figure~\ref{fig.js1313}.
		\begin{figure}[htbp]
			\begin{center}
				
				\tdplotsetmaincoords{-145}{90}
				
				\begin{tikzpicture}[tdplot_main_coords,scale=0.4]
					
					\tdplotsetrotatedcoords{-90}{45}{90}
					
					\begin{scope}[tdplot_rotated_coords]
						\coordinate (O) at (0,0,0);
						\coordinate (P1) at (17,0,0); 
						\coordinate (P2) at (0,17,0); 
						\coordinate (P3) at (0,0,17);
						\coordinate (Q1) at (13,1,3);
						\coordinate (Q2) at (9,2,6);
						\coordinate (Q3) at (5,3,9);
						\coordinate (Q4) at (1,4,12);
						\coordinate (Q5) at (10,6,1);
						\coordinate (Q6) at (6,7,4);
						\coordinate (Q7) at (2,8,7);
						\coordinate (Q8) at (3,12,2);
						
					\end{scope}
					
					\draw (P1) -- (P2) -- (P3) -- cycle; 
					
					\draw (P1) -- (Q4);
					\draw (P1) -- (Q8);
					
					\draw (P3) -- (Q1);
					\draw (P3) -- (Q2);
					\draw (P3) -- (Q3);
					\draw (P3) -- (Q8);
					
					\draw (P2) -- (Q4);
					\draw (P2) -- (Q5);
					\draw (P2) -- (Q7);
					\draw (P2) -- (Q2);
					
					\draw (Q1) -- (Q5);
					\draw (Q1) -- (Q6);
					
					\draw (Q3) -- (Q6);
					\draw (Q3) -- (Q7);
					
					\draw (Q5) -- (Q7);

					
					\node[circle,fill,inner sep=1pt,label={90:$\rho_9(0,0,1)$}] at (P1) {};
					\node[circle,fill,inner sep=1pt,label={0:$\rho_{10}(1,0,0)$}] at (P2) {};
					\node[circle,fill,inner sep=1pt,label={180:$\rho_{11}(0,1,0)$}] at (P3) {};
					\node[circle,fill,inner sep=1pt,label={-90:$\rho_2(\frac{3}{17},\frac{9}{17},\frac{5}{17})$}] at (Q3) {};
					\node[circle,fill,inner sep=1pt,label={-90:$\rho_1(\frac{2}{17},\frac{6}{17},\frac{9}{17})$}] at (Q2) {};
					\node[circle,fill,inner sep=1pt,label={180:$\rho_8(\frac{1}{17},\frac{3}{17},\frac{13}{17})$}] at (Q1) {};
					\node[circle,fill,inner sep=1pt,label={-90:$\rho_3(\frac{4}{17},\frac{12}{17},\frac{1}{17})$}] at (Q4) {};
					\node[circle,fill,inner sep=1pt,label={0:$\rho_7(\frac{6}{17},\frac{1}{17},\frac{10}{17})$}] at (Q5) {};
					\node[circle,fill,inner sep=1pt,label={-90:$\rho_6(\frac{7}{17},\frac{4}{17},\frac{6}{17})$}] at (Q6) {};
					\node[circle,fill,inner sep=1pt,label={-90:$\rho_4(\frac{8}{17},\frac{7}{17},\frac{2}{17})$}] at (Q7) {};
					\node[circle,fill,inner sep=1pt,label={-90:$\rho_5(\frac{12}{17},\frac{2}{17},\frac{3}{17})$}] at (Q8) {};
				\end{tikzpicture}
			\end{center}
			\caption{The junior simplex of $X(1,3,13)$}
			\label{fig.js1313}
		\end{figure}
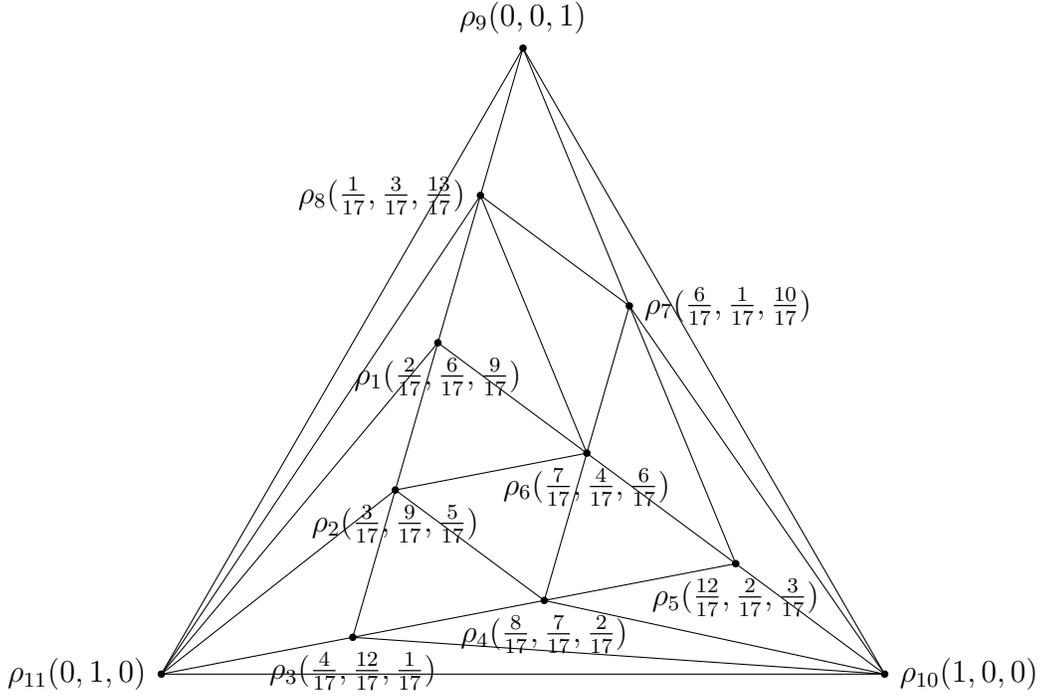
		
		Denote the irreducible compact exceptional surfaces by $S_k\coloneqq V(\rho_k)$ with the closed embedding $i_k\colon S_k\to X(1,3,13)$. Write $C_{k,l}\coloneqq S_k\cap S_l$ with the closed embeddings $j_{k,l}\colon C_{k,l}\to S_l$ and $j_{l,k}\colon C_{k,l}\to S_k$.
		
		Similar to the case of $X(1,3,9)$, using Proposition~\ref{prop.exceptional types}, Lemma~\ref{lemma.sum}, and~\cite[Corollary~1.6]{CR}, we have
		\begin{proposition}\label{prop.exceptionalsurfaces1313}
			All the irreducible compact exceptional surfaces in $X(1,3,13)$ and the self-intersection numbers of their intersection curves are shown in Figure~\ref{fig.rel1313}. In the figure, each hexagon represents an exceptional surface $S_k$, each edge represents a $\mathbb{P}^1$ intersection curve $C_{k,l}$, and the red number adjacent to an edge inside hexagon $S_k$ is $C_{k,l}^2$ in $S_k$. A vertex where three hexagons meet indicates a triple intersection point, as shown in Figure~\ref{fig.rel139}.
			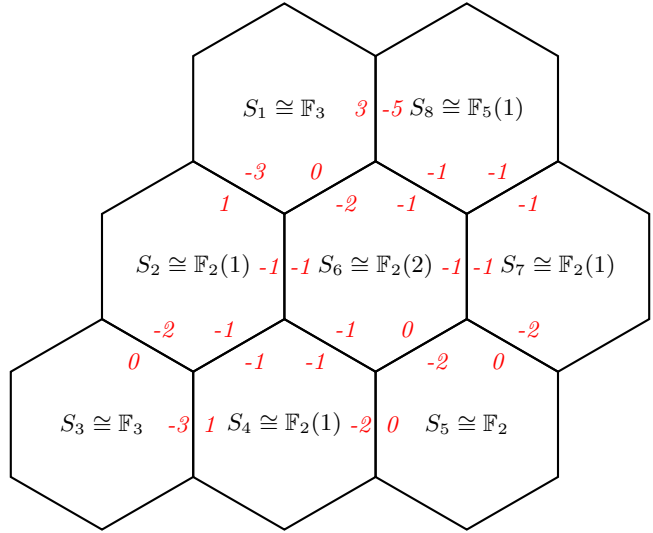
\begin{figure}[htbp]
				\centering
				\begin{tikzpicture}[scale=1.4,font=\scriptsize]
					\hexsetup
					\hexagon{-\hs}{0}{S_1\cong\F_3}
					\hexagon{\hs}{0}{S_8\cong\F_5(1)}
					\hexagon{-2*\hs}{-1.5}{S_2\cong\F_2(1)}
					\hexagon{0}{-1.5}{S_6\cong\F_2(2)}
					\hexagon{2*\hs}{-1.5}{S_7\cong\F_2(1)}
					\hexagon{-3*\hs}{-3}{S_3\cong\F_3}
					\hexagon{-\hs}{-3}{S_4\cong\F_2(1)}
					\hexagon{\hs}{-3}{S_5\cong\F_2}
					\node[fill=white, inner sep=1pt, text=red] at (-8/3*\hs,-2.4) {0};  
					\node[fill=white, inner sep=1pt, text=red] at (-7/3*\hs,-2.1) {-2};
					\node[fill=white, inner sep=1pt, text=red] at (-0.15-2*\hs,-3) {-3};  
					\node[fill=white, inner sep=1pt, text=red] at (0.15-2*\hs,-3) {1};
					\node[fill=white, inner sep=1pt, text=red] at (-5/3*\hs,-0.9) {1};  
					\node[fill=white, inner sep=1pt, text=red] at (-4/3*\hs,-0.6) {-3};
					\node[fill=white, inner sep=1pt, text=red] at (-0.15-\hs,-1.5) {-1};  
					\node[fill=white, inner sep=1pt, text=red] at (0.15-\hs,-1.5) {-1};
					\node[fill=white, inner sep=1pt, text=red] at (-4/3*\hs,-2.4) {-1};  
					\node[fill=white, inner sep=1pt, text=red] at (-5/3*\hs,-2.1) {-1};
					\node[fill=white, inner sep=1pt, text=red] at (-1/3*\hs,-0.9) {-2};  
					\node[fill=white, inner sep=1pt, text=red] at (-2/3*\hs,-0.6) {0};
					\node[fill=white, inner sep=1pt, text=red] at (-2/3*\hs,-2.4) {-1};  
					\node[fill=white, inner sep=1pt, text=red] at (-1/3*\hs,-2.1) {-1};
					\node[fill=white, inner sep=1pt, text=red] at (2/3*\hs,-2.4) {-2};  
					\node[fill=white, inner sep=1pt, text=red] at (1/3*\hs,-2.1) {0};
					\node[fill=white, inner sep=1pt, text=red] at (-0.15,-3) {-2};  
					\node[fill=white, inner sep=1pt, text=red] at (0.15,-3) {0};
					\node[fill=white, inner sep=1pt, text=red] at (-0.15,0) {3};  
					\node[fill=white, inner sep=1pt, text=red] at (0.15,0) {-5};
					\node[fill=white, inner sep=1pt, text=red] at (1/3*\hs,-0.9) {-1};  
					\node[fill=white, inner sep=1pt, text=red] at (2/3*\hs,-0.6) {-1};
					\node[fill=white, inner sep=1pt, text=red] at (-0.15+\hs,-1.5) {-1};  
					\node[fill=white, inner sep=1pt, text=red] at (0.15+\hs,-1.5) {-1};
					\node[fill=white, inner sep=1pt, text=red] at (4/3*\hs,-2.4) {0};  
					\node[fill=white, inner sep=1pt, text=red] at (5/3*\hs,-2.1) {-2};
					\node[fill=white, inner sep=1pt, text=red] at (4/3*\hs,-0.6) {-1};  
					\node[fill=white, inner sep=1pt, text=red] at (5/3*\hs,-0.9) {-1};
				\end{tikzpicture}
				\caption{Exceptional surfaces in $X(1,3,13)$ and their intersections}
				\label{fig.rel1313}
			\end{figure}
		\end{proposition}
		\begin{proof}
			As in the proof of Proposition~\ref{prop.excsurfin139}, denote the subfan generated by $\{\rho_2,\rho_3,\rho_4,\rho_5,\rho_6,\rho_8\}$ as $\Sigma_4$. Then there is a lattice isomorphism $\phi\colon L_{(1,3,13)}\to \Z^3$, such that $\Sigma_4$ and $\Sigma_{\F_2(1)}\times K_{\F_2(1)}$ are compatible. Thus, $S_4\cong\F_2(1)$ and $C_{4,k}\cong V(v_k)$ for $k=2,3,5,6,8$, where $v_k$ is the 1-dimensional cone in $\Sigma_{\F_2(1)}$ as shown in Figure~\ref{fig.H2(1)}. 
			Thus 
			\[C_{24}^2=C_{46}^2=-1,\qquad
			C_{45}^2=-2,\qquad
			C_{34}^2=1,\]
			in $S_4$.
			Then the remaining exceptional surfaces and intersection curves follow from Lemma~\ref{lemma.sum} and Proposition~\ref{prop.exceptional types}.
		\end{proof}
		\subsection{The group action on $D(X(1,3,13))$}
		
		Denote $\cE_{2l-1}\coloneqq i_{2l-1*}\cO_{S_{2l-1}}$ and $\cE_{2l}\coloneqq i_{2l*}\cO_{S_{2l}}(-C_{2l-1,2l})$ for $l=1,2,3$. Denote $\cE_t=i_{t*}\cO_{S_t}\otimes \cO_{X(1,3,13)}(S_7+S_8)$ for $t=7,8$. Thus, using Proposition~\ref{prop.sph} and Lemma~\ref{lemma.sph}, each $\cE_k$ is spherical in $D(X(1,3,13))$. Denote $T_k\coloneqq T_{\cE_k}$, the spherical twist associated with $\cE_k$. This yields an algebraic braid twist group action on $D(X(1,3,13))$.
		
		\begin{theorem}\label{thm.E8}
			The algebraic braid twist group $\AT(Q',W')$ associated with the quiver $(Q',W')$ in Figure~\ref{fig.quiver1313} acts on $D(X(1,3,13))$ faithfully via $\beta_k\mapsto T_k$.
			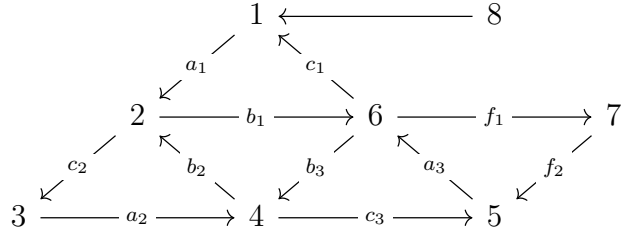
\begin{figure}[htbp]
				\centering
				\begin{tikzcd}
					&& 1 && 8 \\
					& 2 && 6 && 7 \\
					3 && 4 && 5
					\arrow["{a_1}"{description}, from=1-3, to=2-2]
					\arrow[from=1-5, to=1-3]
					\arrow["{b_1}"{description}, from=2-2, to=2-4]
					\arrow["{c_2}"{description}, from=2-2, to=3-1]
					\arrow["{c_1}"{description}, from=2-4, to=1-3]
					\arrow["{f_1}"{description}, from=2-4, to=2-6]
					\arrow["{b_3}"{description}, from=2-4, to=3-3]
					\arrow["{f_2}"{description}, from=2-6, to=3-5]
					\arrow["{a_2}"{description}, from=3-1, to=3-3]
					\arrow["{b_2}"{description}, from=3-3, to=2-2]
					\arrow["{c_3}"{description}, from=3-3, to=3-5]
					\arrow["{a_3}"{description}, from=3-5, to=2-4]
				\end{tikzcd}
				\caption{$(Q',W'=\sum_1^3(-a_kb_kc_k)+b_3b_2b_1+a_3f_1f_2)$}
				\label{fig.quiver1313}
			\end{figure}
		\end{theorem}
		
		As in the case $X(1,3,9)$, we have
		\begin{proposition}\label{prop.D6con}
			$\{\cE_k\}_{1\leq k\leq 6}$ is a $(Q,W)$-configuration in $D(X(1,3,13))$.
		\end{proposition}
		
		In addition,
		\begin{proposition}\label{prop.1313A2}
			We have
			\begin{itemize}
				\item[(1)] $\dHomk(\cE_8,\cE_1)=\C[-1]$;
				\item[(2)] $\dHomk(\cE_8,\cE_6)=0$;
				\item[(3)] $\dHomk(\cE_8,\cE_7)=0$;
				\item[(4)] $\dHomk(\cE_6,\cE_7)=\C[-1]$;
				\item[(5)] $\dHomk(\cE_7,\cE_5)=\C[-1]$.
			\end{itemize}
		\end{proposition}
		\begin{proof}
			(1) and (5) follow from Corollary~\ref{coro.Hom}, since $S_7\cap S_1=S_8\cap S_5=\varnothing$.
			
			(2) follows from Corollary~\ref{coro.Homspace}, 
			\[\dHomk(\cE_8,\cE_6)\cong \Gamma_{C_{68}}^\bullet\cO_{C_{68}}(S_8-S_7-S_8-S_5)[-1]\cong \Gamma_{C_{68}}^\bullet\cO_{C_{68}}(-1)[-1]=0,\]
			where the last isomorphism uses $S_5\cap C_{68}=\varnothing$ and  $S_7\cap C_{68}=\{\text{a point}\}$.
			
			(3) follows from Proposition~\ref{prop.otimes},
			\[\dHomk(\cE_8,\cE_7)\cong \Gamma_{C_{78}}^\bullet\cO_{C_{78}}(S_8-S_8-S_7+S_8+S_7)[-1]\cong \Gamma_{C_{78}}^\bullet\cO_{C_{78}}(-1)[-1]=0,\] 
			where the last isomorphism uses Lemma~\ref{lemma.number} and Proposition~\ref{prop.exceptionalsurfaces1313}.
			
			(4) follows from Corollary~\ref{coro.Homspace},
			\[\dHomk(\cE_6,\cE_7)\cong \Gamma_{C_{67}}^\bullet\cO_{C_{67}}(S_6+S_7+S_5+S_8)[-1]\cong \Gamma_{C_{67}}^\bullet\cO_{C_{67}}[-1]=\C[-1],\]
			where the last isomorphism uses Lemma~\ref{lemma.sum}, $S_8\cap C_{67}=\{\text{a point}\}$ and $S_5\cap C_{67}=\{\text{a point}\}$.
		\end{proof}

		Denote $S_{56}\coloneqq  S_5\cup S_6$ and the closed embedding by $i_{56}\colon S_{56}\to X(1,3,13)$. Denote $C_7=C_{57}\cup C_{67}$ and the closed embedding by $j_7\colon  C_7\to S_{56}$.
		
		\begin{proposition}\label{prop.1313xyzperp(1)}
			$T_5\cE_6$ lies in the right orthogonal of $\cE_7$, i.e.,~$T_5\cE_6\in\cE_7^\perp$.
		\end{proposition}
		\begin{proof}
			Using Proposition~\ref{prop.exceptionalsurfaces1313} and Lemma~\ref{lemma.extcon}, we can directly get
			\[T_5\cE_6\cong {i_{56*}}\cO_{S_{56}}.\]
			
			By Lemma~\ref{lemma.homexten}(1),
			\[\mathcal{H}om_{X(1,3,13)}(\cE_7, T_5\cE_6)
			\cong  i_{56*}j_{7*}\cK[-1]\]
			where $\cK=\mathcal{O}_{C_7}(S_7-S_7-S_8)\otimes j_7^*\cM$ and $\cM= \cO_{S_{56}}$. 
			
			Using Lemma~\ref{lemma.homexten}(2), we have \(\cK|_{C_{57}}\cong \mathcal{O}_{C_{57}}\), 
			
			Using Lemma~\ref{lemma.homexten}(3), we have \(\cK|_{C_{67}}\cong \mathcal{O}_{C_{67}}(S_5-2)\cong \mathcal{O}_{C_{67}}(-1)\), because $S_5\cap C_{67}=\{\text{a point}\}$.
			
			Thus, the claim follows from Lemma~\ref{lemma.perpgen}.
		\end{proof}
		
		\begin{proposition}\label{prop.E8con}
			\(\{\cE_1,\cE_2,\cE_3,\cE_4,\cE_5,\cE_6,\cE_7,\cE_8\}\) is a $(Q',W')$-configuration in $D(X(1,3,13))$. 
		\end{proposition}
		\begin{proof}
			This claim follows from Proposition~\ref{prop.D6con},~\ref{prop.1313A2} and~\ref{prop.1313xyzperp(1)}.
		\end{proof}
		
		\begin{proposition}\label{prop.E8} \[\{\cF_1\coloneqq\cE_5,\,\cF_2\coloneqq T_4T_5T_2(\cE_6),\,\cF_3\coloneqq\cE_4,\,\cF_4\coloneqq\cE_3,\,\cF_5\coloneqq T_2(\cE_3),\, \cF_6\coloneqq \cE_1,\,\cF_7\coloneqq\cE_7,\,\cF_8\coloneqq\cE_8\}\] is an $E_8$-configuration as shown in Figure~\ref{fig.E8con}.
			\begin{figure}[htbp]
				\centering
				\begin{tikzcd}
					\cF_7 & \cF_1 & \cF_3 & \cF_4 & \cF_5 & \cF_6 & \cF_8 \\
					&& \cF_2
					\arrow[from=1-1, to=1-2]
					\arrow[from=1-2, to=1-3]
					\arrow[from=1-3, to=1-4]
					\arrow[from=1-4, to=1-5]
					\arrow[from=1-5, to=1-6]
					\arrow[from=1-6, to=1-7]
					\arrow[from=2-3, to=1-3]
				\end{tikzcd}
				\caption{$E_8$-configuration}
				\label{fig.E8con}
			\end{figure}
		\end{proposition}
		\begin{proof}
			By Proposition~\ref{prop.D6}, \[\{\cF_1\coloneqq\cE_5,\,\cF_2\coloneqq T_4T_5T_2(\cE_6),\,\cF_3\coloneqq\cE_4,\,\cF_4\coloneqq\cE_3,\,\cF_5\coloneqq T_2(\cE_3),\, \cF_6\coloneqq \cE_1\}\] is a $D_6$-subconfiguration in $D(X(1,3,13))$. 
			Using Proposition~\ref{prop.E8con}, there is no arrow between $7$ and $k$ in $Q'$, for $k=1,2,3,4,8$.
			Thus
			\[\dHomk(\cF_7,\cF_l)=0,\quad\text{for }l=3,4,5,6,8;\]
			and
			\[\dHomk(\cF_7,\cF_2)=\dHomk(\cE_7,T_4T_5T_2(\cE_6))\cong\dHomk(\cE_7,T_5(\cE_6))=0,\]
			where the isomorphism is because there is no arrow between $2$ and $5$ in $W'$ and the final equality is because there is a cycle $5\to6\to7\to5$ in $W'$.
			
			Using Proposition~\ref{prop.1313A2}(5), we have
			\[\dHomk(\cF_7,\cF_1)=\dHomk(\cE_7,\cE_5)\cong\C[-1].\]
			
			Again by Proposition~\ref{prop.E8con}, there is no arrow between $8$ and $k$ in $Q'$, for $2\leq k\leq 7$.
			Thus
			\[\dHomk(\cF_8,\cF_l)=0,\quad\text{for }l=1,2,3,4,5,7.\]
			
			Again by Proposition~\ref{prop.1313A2}(1), we have
			\[\dHomk(\cF_8,\cF_6)=\dHomk(\cE_8,\cE_1)\cong\C[-1].\]
			
			Thus, the claim follows from the above $\Hom$ computations together with Serre duality.
			
		\end{proof}

		\begin{proof}[Proof of Theorem~\ref{thm.E8}]
			By Proposition~\ref{prop.E8con}, $\{\cE_k\}_{1\leq k\leq 8}$ is a $(Q',W')$-configuration in $D(X(1,3,13))$.
			Thus, the action via $\beta_k\mapsto T_k$ is given by Proposition~\ref{prop.cycle}.
			
			To prove faithfulness, by Proposition~\ref{prop.E8}, \[\{\cF_1\coloneqq\cE_5,\,\cF_2\coloneqq T_4T_5T_2(\cE_6),\,\cF_3\coloneqq\cE_4,\,\cF_4\coloneqq\cE_3,\,\cF_5\coloneqq T_2(\cE_3),\, \cF_6\coloneqq \cE_1,\,\cF_7\coloneqq\cE_7,\,\cF_8\coloneqq\cE_8\}\] is an $E_8$-configuration in $D(X(1,3,13))$ as shown in Figure~\ref{fig.E8con}. Using \cite[Theorem~1]{NV}, the $E_8$-configuration induces a faithful $\Br(E_8)$-action via $T'_k \coloneqq T_{\cF_k}$. For $1 \leq k \leq 6$, the $T'_k$ have the same expressions in terms of $\{T_l\}_{1\leq l\leq 6}$ as in the proof of Theorem~\ref{thm.D6}, and for $t=7,8$ we have $T'_t = T_t$. Thus, arguing as in the proof of Theorem~\ref{thm.D6}, each $T_l$ (for $1 \leq l \leq 8$) can be expressed in terms of the $T'_k$s.
			Hence, the action of $\AT(Q',W')$ is also faithful.
		\end{proof}

\end{sloppypar}
\end{document}